\documentclass[10pt]{article}
\usepackage{diagbox}
\usepackage{lmodern}
\usepackage{amsmath}
\usepackage[T1]{fontenc}
\usepackage[utf8]{inputenc}
\usepackage{authblk}
\usepackage{amsfonts}
\usepackage{graphicx}
\usepackage{rotating}
\usepackage{amssymb}
\usepackage[english]{babel}
\usepackage{color}
\usepackage{amsthm}
\usepackage{graphicx}
\usepackage{tkz-euclide}
\tikzset{elegant/.style={smooth,thick,samples=50,cyan}}
\usepackage{color,tikz}
\usetikzlibrary{patterns}
\usetikzlibrary{decorations.pathreplacing,calc}
\usepackage{rotating}
\usepackage{mathrsfs}
\usepackage{makecell}
\usepackage{microtype}
\usepackage{enumerate}
\usepackage{mathscinet}
\usepackage{array}
\usepackage{multirow}
\usepackage{booktabs}
\usepackage{enumerate}
\usepackage[cal=boondoxo,bb=ams]{mathalfa}
\usepackage{hyperref}
\hypersetup{hidelinks}
\usepackage{titlesec}
\newtheorem{theorem}{Theorem}[section]
\newtheorem{defn}{Definition}[section]
\newtheorem{prop}{Proposition}[section]
\newtheorem{lemma}{Lemma}[section]
\newtheorem{coro}{Corollary}[section]
\newtheorem{remark}{Remark}[section]

\newcommand{\ml}{\mathcal}
\newcommand{\mb}{\mathbb}

\DeclareMathOperator{\lin}{lin}
\DeclareMathOperator{\non}{non}
\DeclareMathOperator{\intt}{int}
\DeclareMathOperator{\extt}{ext}
\DeclareMathOperator{\bdd}{bdd}
\DeclareMathOperator{\divv}{div}

\title{The influence of viscous dissipations on the nonlinear acoustic wave equation with second sound}
\author[1]{Wenhui Chen\thanks{Wenhui Chen (wenhui.chen.math@gmail.com)}}
\affil[1]{School of Mathematics and Information Science, Guangzhou University, 510006 Guangzhou, China}
\author[2]{Yan Liu\thanks{Yan Liu (ly801221@163.com)}}
\affil[2]{Department of Applied Mathematics, Guangdong University of Finance, 510521 Guangzhou, China}
\author[3]{Alessandro Palmieri\thanks{Alessandro Palmieri (alessandro.palmieri.math@gmail.com)}}
\affil[3]{Department of Mathematics, University of Bari, Via E. Orabona 4, Bari 70125, Italy}

\author[4]{Xulong Qin\thanks{Xulong Qin (qinxul@mail.sysu.edu.cn)}}
\affil[4]{Department of Mathematics, Sun Yat-Sen University, 510275 Guangzhou, China}

\date{}
\begin{document}

\maketitle

\begin{abstract}
	\medskip
We study the effect of a viscous dissipation on the Cauchy problem for a Cattaneo-type model in nonlinear acoustics, established by applying the Lighthill approximation for the viscous or inviscid fluid model. The contribution of this paper is twofold. For the nonlinear viscous Cattaneo-type model involving a fractional Laplacian $(-\Delta)^{\alpha}$ in the viscous damping with $\alpha\in[0,1]$, we derive optimal decay rates for global (in time) solutions with small data in certain Sobolev spaces. Furthermore, by introducing a threshold $\alpha=1/2$ for the power of the fractional viscous dissipation, we derive an anomalous diffusion profile when $\alpha\in[0,1/2)$ and a diffusion wave profile when $\alpha\in[1/2,1]$ for large-time. Whereas, for the nonlinear inviscid Cattaneo-type model (or the Jordan-Moore-Gibson-Thompson equation in the critical case), we obtain the blow-up of the energy solutions in finite time under  suitable assumptions for the initial data. Thus, the presence of a viscous dissipation in the nonlinear Cattaneo-type model is a criterion for the global (in time) existence and blow-up of solutions.
 \\
	
	
	\noindent\textbf{AMS Classification (2020)}  35L05, 35R11, 35A01, 35B40, 35B44.

\end{abstract}
\fontsize{12}{15}
\selectfont
\section{Introduction}
In the classical  theory of acoustic  waves, in order to describe the propagation of sound in thermoviscous fluids, some mathematical models for the classification of nonlinear wave equations with   dissipations have been considered in the study of nonlinear acoustics, e.g. the Kuznetsov equation, the Westervelt equation and the Blackstock model. We address the interested reader  to the reviews and introductions in \cite{Kaltenbacher=2015,Kaltenbacher-Thalhammer-2018}. These models have been widely used in medical and industrial applications of high-intensity ultra sound, for instance, medical imaging and therapy, ultrasound cleaning and welding (see \cite{Abramov-1999,Dreyer-Krauss-Bauer-Ried-2000,Kaltenbacher-Landes-Hoffelner-Simkovics-2002} and references therein).
\subsection{Fundamental models in nonlinear acoustics}
In thermoviscous fluids, the PDEs of nonlinear acoustics are governed by the conservations of mass, momentum and energy, supplemented with the heuristic equation of state for perfect gases. We recall the well-known fully compressible Navier-Stokes equations associated with suitable heat conductions as follows (cf. \cite{BCD=2011,L=1998,T}):
\begin{align}\label{CNS}
\begin{cases}
\displaystyle{\rho_t+\divv (\rho\mathbf{u})=0,}\\
\displaystyle{(\rho\mathbf{u})_t+\mathbf{u}\divv (\rho\mathbf{u})+\rho(\mathbf{u}\cdot\nabla)\mathbf{u}+\nabla p-\mu_{\mathrm{V}}\Delta\mathbf{u}-\left(\mu_{\mathrm{B}}+\frac{1}{3}\mu_{\mathrm{V}}\right)\nabla\divv\mathbf{u}=0,}\\
\displaystyle{\rho\left(c_{\mathrm{V}}\Theta_t+c_{\mathrm{V}}\mathbf{u}\cdot\nabla \Theta+\frac{c_{\mathrm{P}}-c_{\mathrm{V}}}{\alpha_{\mathrm{V}}}\divv\mathbf{u}\right)}\\
\displaystyle{\qquad\qquad=-\divv \mathbf{q}+\left(\mu_{\mathrm{B}}-\frac{2}{3}\mu_{\mathrm{V}}\right)|\divv\mathbf{u}|^2+\frac{1}{2}\mu_{\mathrm{V}}\left\|\nabla\mathbf{u}+(\nabla\mathbf{u})^{\mathrm{T}}\right\|_{F}^2,}\\
\displaystyle{p=R\rho\, \Theta,}
\end{cases}
\end{align}
with the Frobenius norm $\|\cdot\|_F$, where the physical variables and quantities are introduced in the next tables. Later, the equation of the conduction for the heat flux $\mathbf{q}$ will be stated (cf. Section \ref{Subsection background}).
\medskip

\begin{minipage}{\textwidth}
	\begin{minipage}[t]{0.48\textwidth}
		\makeatletter\def\@captype{table}
		\begin{tabular}{ll} 
			\toprule
			Physical variable & Notation   \\
			\midrule 
			Mass density & $\rho$\\
			Acoustic particle velocity & $\mathbf{u}$\\
			Acoustic velocity potential & $\psi$\\
			Acoustic pressure & $p$\\
			Absolute temperature & $\Theta$\\
			Heat flux & $\mathbf{q}$\\
			\bottomrule
		\end{tabular}
		\caption{Notations for the physical variables}
		\label{Table_1}
	\end{minipage}
	\begin{minipage}[t]{0.48\textwidth}
		\makeatletter\def\@captype{table}
		\begin{tabular}{ll} 
			\toprule
			Physical quantity & Notation   \\
			\midrule 
			Bulk/Shear viscosity  & $\mu_{\mathrm{B}}$, $\mu_{\mathrm{V}}$\\
			Kinematic viscosity & $\nu$\\
					Viscosity number & $b$\\
			Specific heat (volume/pressure) & $c_{\mathrm{V}}$, $c_{\mathrm{P}}$\\
			Thermal expansion coefficient &$\alpha_{\mathrm{V}}$\\
			Thermal conductivity & $\kappa$\\
			Thermal relaxation & $\tau$\\
			Speed of sound & $c_0$\\
			Specific gas constant & $R$\\
			Parameters of nonlinearities & $A$, $B$\\
			\bottomrule
		\end{tabular}
		\caption{Notations for the physical quantities}
		\label{Table_2}
	\end{minipage}	
\end{minipage}
\medskip

\noindent Then, we will apply the Lighthill approximation \cite{Lighthill-1956,Blackstock-1963} to the nonlinear system \eqref{CNS}, where we only retain the terms of first- and second-order with small perturbations around the equilibrium state. Considering irrotational flows $\nabla\times\mathbf{u}=0$, which indicates $\mathbf{u}=-\nabla\psi$, several mathematical models in nonlinear acoustics have been deduced over the last century. One may regard nonlinear viscous acoustic wave models as approximations of the fully compressible Navier-Stokes equations  under irrotational flows. Additionally, for inviscid fluids, the nonlinear system \eqref{CNS} with $\mu_{\mathrm{B}}=\mu_{\mathrm{V}}=0$  leads to nonlinear inviscid acoustic wave models by some approximations of the fully compressible Euler equations under irrotational flows.

\subsubsection{Background of nonlinear viscous acoustic wave models} \label{Subsection background}
For the sake of brevity, we introduce the  following notation for the nonlinearity
\begin{align}\label{N0}
	\ml{N}_{\psi}(t,x):=\frac{B}{2Ac_0^2}|\psi_t(t,x)|^2+|\nabla\psi(t,x)|^2.
\end{align} 
 By combining Fourier's law for the heat conduction
\begin{align}\label{Fourier-Law}
\mathbf{q}=-\kappa\nabla \Theta,
\end{align}
with the compressible Navier-Stokes-Fourier equations \eqref{CNS}, and applying the Lighthill scheme, we may derive the well-known Kuznetsov equation (see \cite{Kuznetsov-1971,Mizohata-Ukai=1993,Fritz-Nikolic-Wohlmuth=2018,Kaltenbacher-Nikolic=2022} and references therein)
\begin{align}\label{Kuznetsov-Eq}
\psi_{tt}-c_0^2\Delta\psi-b\nu\Delta \psi_t=\partial_t\ml{N}_{\psi},
\end{align}
where the thermal conductivity was assumed to be small in comparison with the size of small perturbations in the Lighthill approximation. 
From the Kuznetsov equation \eqref{Kuznetsov-Eq} in classical nonlinear acoustics, an infinite signal speed paradox occurs due to the application of Fourier's law \eqref{Fourier-Law}. It seems to be unrealistic  in the acoustic waves propagation.  To eliminate this drawback, motivated by the second sound theory \cite{Chandrasekharaiah}, one replaces Fourier's law \eqref{Fourier-Law} by Cattaneo's law
\begin{align}\label{Cattaneo-Law}
\tau\mathbf{q}_t+\mathbf{q}=-\kappa\nabla \Theta.
\end{align}
 Because the presence of a finite propagation speed for acoustic waves, the models derived by using Cattaneo's law are  significant and well-accepted. A remarkable example of these models is given by the Jordan-Moore-Gibson-Thompson (JMGT) equation (see \cite{Kaltenbacher-Lasiecka-Marchand-2011,Kaltenbacher-Lasiecka-Pos-2012}).

For the above reasons, in the recent work \cite{Liu-Qin-Zhang=2022} (more precisely, see \cite[Section 2]{Liu-Qin-Zhang=2022} with the small thermal conductivity $\varepsilon^2\kappa$ carrying the size $\varepsilon$ of small perturbations) it is derived the following model: 
\begin{align}\label{Eq-New-C}
(\tau\partial_t+\ml{I})(\psi_{tt}-c_0^2\Delta\psi-b\nu\Delta \psi_t)=\partial_t\ml{N}_{\psi},
\end{align}
where $\ml{I}$ denotes the identity operator. Roughly speaking, \eqref{Eq-New-C} is the counterpart of \eqref{Kuznetsov-Eq} in the theory of hyperbolic heat counduction. Clearly, the Cattaneo-type model \eqref{Eq-New-C} is reduced to the Fourier-type model \eqref{Kuznetsov-Eq} as $\tau=0$.
By a Cattaneo-type model, we mean a model in nonlinear acoustics which is derived by using Cattaneo's law of heat conduction in place of Fourier's law. This viscous Cattaneo-type model with $b\nu>0$ is obtained by the Lighthill approximation of the compressible Navier-Stokes-Cattaneo equations \eqref{CNS} combined with the heat conduction \eqref{Cattaneo-Law}. Note that the viscous dissipation $-b\nu\Delta\psi_t$ in the nonlinear acoustic wave equations \eqref{Kuznetsov-Eq} and \eqref{Eq-New-C} only comes from the viscosity in the Navier-Stokes equation \eqref{CNS}$_2$ such that
\begin{align}\label{Viscosity}
-\mu_{\mathrm{V}}\Delta\mathbf{u}-\left(\mu_{\mathrm{B}}+\frac{1}{3}\mu_{\mathrm{V}}\right)\nabla\divv\mathbf{u}=-\left(\mu_{\mathrm{B}}+\frac{4}{3}\mu_{\mathrm{V}}\right)\Delta\mathbf{u}=-b\nu\Delta\mathbf{u}
\end{align}
due to the irrotational condition $\nabla\times\mathbf{u}=0$. A change in the viscosity in the Navier-Stokes equations produces in turn a modification in the viscous dissipation of the nonlinear acoustics model \eqref{Eq-New-C}.

\subsubsection{Background of nonlinear inviscid acoustic wave models}
Let us turn to the inviscid case in nonlinear acoustics. By employing the Lighthill approximation to the nonlinear inviscid system \eqref{CNS} with the Cattaneo's law \eqref{Cattaneo-Law}, we may establish the well-known inviscid Cattaneo-type model (aslo known as the  JMGT equation in the critical case cf. \cite{Kaltenbacher-Lasiecka-Marchand-2011,Dell-Lasiecka-Pata,Nikolic-Said=2021} and references therein)
\begin{align}\label{Inviscid-JMGT}
(\tau\partial_t+\ml{I})(\psi_{tt}-c_0^2\Delta\psi)=\partial_t\ml{N}_{\psi},
\end{align}
which is also the inviscid case for  \eqref{Eq-New-C} with $b\nu=0$ (recall that, according to \eqref{Viscosity}, $b\nu=\mu_{\mathrm{B}}+\frac{4}{3}\mu_{\mathrm{V}}$). This inviscid Cattaneo-type model is the Lighthill approximation of the compressible Euler-Cattaneo equations \eqref{CNS} with $\mu_{\mathrm{B}}=\mu_{\mathrm{V}}=0$ and the heat conduction \eqref{Cattaneo-Law}. The lack of a viscous dissipation in  the linearized model leads to the conservation of a suitable defined energy (see \cite{Kaltenbacher-Lasiecka-Marchand-2011,Marchand-McDevitt-Triggiani=2012}) rather than 
to the decay of the energy. It arises quite naturally the question on whether or not local in time solutions to the nonlinear inviscid Cattaneo-type model \eqref{Inviscid-JMGT} may be globally extended.  We will give a negative answer to this question for low dimensions, namely, by showing that energy solutions to \eqref{Inviscid-JMGT} in $\mb{R}^n$ blow up in finite time under some integral conditions for the initial data when $n=1,2,3$. Note that the limit case $\tau=0$ with $\ml{N}_{\psi}=|\psi_t|^2$ has been studied by \cite{Alinhac}.

\subsection{Nonlinear acoustics model involving the fractional viscous dissipations}
It is well-known that the fractional Laplace operator may simulate various anomalous diffusions with wide applications (see two recent monographs \cite{Bucur-Valdinoci=2016,Pozrikidis=2016}). In the last two decades, the Navier-Stokes equations with a  fractional viscous term $(-\Delta)^{\alpha}\mathbf{u}$, which have been employed to describe the motion of a fluid with internal friction interactions (see \cite{Mercado-Guido=2013}), have been widely studied. One may see the pioneering monograph \cite{Lions=1969} in 1969, and related references \cite{Katz-Pavlovic=2002,Wu=2006,Tao=2009,Tang-Yu=2015,Wang-Wu-Ye=2020}  for the fractional Navier-Stokes equations. 

Motivated by the recent progresses for these fractional Navier-Stokes equations, according to the viscosity expression \eqref{Viscosity} for an irrotational flow, it seems reasonable to replace the classical viscosity $-b\nu\Delta\mathbf{u}$ with the fractional viscosity $b\nu(-\Delta)^{\alpha}\mathbf{u}$. Let us follow a standard approach to deduce nonlinear acoustics models, e.g. \cite[Section 2]{Bucur-Valdinoci=2016}, \cite[Appendix A]{Kaltenbacher-Thalhammer-2018} and \cite[Section 2]{Liu-Qin-Zhang=2022}. To be specific, by combining the conservation of mass, the conservation of momentum involved the fractional Laplacians, the conservation of energy associated with Cattaneo's law \eqref{Cattaneo-Law}, and the state equation under irrotational flows, via the Lighthill approximation, we may derive the following Cattaneo-type model with fractional Laplacians of the viscosity:
\begin{align}\label{Eq-New-C-2}
	(\tau\partial_t+\ml{I})(\psi_{tt}-c_0^2\Delta\psi+b\nu(-\Delta)^{\alpha} \psi_t)=\partial_t\ml{N}_{\psi},
\end{align}
%
%
where the physical constants $\tau,b,\nu,c_0,A,B$ are explained in Table \ref{Table_1} and Table \ref{Table_2}. In other words, the viscous dissipation in the nonlinear acoustic wave equation is determined by the viscous terms in the conservation of momentum \eqref{CNS}$_2$ so that the derivation of \eqref{Eq-New-C-2} is just a slight modification of the one in \cite[Section 2]{Liu-Qin-Zhang=2022} via replacing $\varepsilon\kappa$ by $\varepsilon^2\kappa$, and further replacing the viscous term in \cite[Equation (11)]{Liu-Qin-Zhang=2022} by the term with the fractional Laplacian $-\varepsilon\sqrt{\varepsilon}(\mu_{\mathrm{B}}+\frac{4}{3}\mu_{\mathrm{V}})(-\Delta)^{\alpha}\psi$.  

In particular, taking $\alpha=1$, our model in \eqref{Eq-New-C-2} coincides exactly  with the classical Cattaneo-type model in \eqref{Eq-New-C}. Thus, the non-local third-order (in time) evolution equation \eqref{Eq-New-C-2} can be regarded as the Lighthill approximation of the fractional Navier-Stokes-Cattaneo equations under irrotational flows. However, differently from the classical Cattaneo-type model \eqref{Eq-New-C}, the fractional dissipation $b\nu(-\Delta)^{\alpha}\psi_t$ may completely change the behavior of solutions depending on the power $\alpha$ in regards to the global (in time) strong solvabilities in the fractional Navier-Stokes equations \cite{Katz-Pavlovic=2002}, and the effectiveness of the damping (see \cite{D'Abbicco-Reissig=2014,D'Abbicco-Ebert=2014-NA,D'Abbicco-Ebert=2022-NA}). Hence, it is interesting to explore the influence of the fractional viscous dissipation on the solutions to \eqref{Eq-New-C-2} asymptotically for long-time.

\subsection{Main purpose of the paper}
In the present paper, we investigate the effect of the viscous dissipation on the following Cauchy problem for the Cattaneo-type model:
\begin{align}\label{Eq-Cattaneo-Fractional}
\begin{cases}
(\tau\partial_t+\ml{I})\left(\psi_{tt}-c_0^2\Delta\psi+b\nu(-\Delta)^{\alpha} \psi_t\right)=\partial_t\ml{N}_{\psi},&x\in\mb{R}^n,\ t>0,\\
	\psi(0,x)=\psi_0(x),\ \psi_t(0,x)=\psi_1(x),\ \psi_{tt}(0,x)=\psi_2(x),&x\in\mb{R}^n,
\end{cases}
\end{align}
where the nonlinearity $\ml{N}_{\psi}=\ml{N}_{\psi}(t,x)$ is defined in \eqref{N0}, $\ml{I}$ is the identity operator and $\alpha\in [0,1]$, $\tau>0$, $c_0>0$, $b\nu\geqslant0$, $B/A>0$ are physical constants as in Table \ref{Table_2}. The contribution of this paper is twofold. More precisely, we will consider the small viscosity case $0<b\nu\ll 1$ and the inviscid case $b\nu=0$, respectively. Indeed, determining the behavior  of viscous flows at small viscosity is one of the fundamental problems in fluid mechanics (see, for example, \cite{Maekawa-Mazz=2018}). Our main purposes 
in this work are
\begin{enumerate}
\item to derive the long-time asymptotic behavior of global (in time) solutions to the viscous Cattaneo-type model \eqref{Eq-Cattaneo-Fractional} for $0<b\nu\ll 1$ under the influence of a fractional Laplacian, emphasizing the role of the power $\alpha\in[0,1]$; 
\item to prove the blow-up in finite time for energy solutions to the inviscid Cattaneo-type model in \eqref{Eq-Cattaneo-Fractional} when $b\nu=0$.
\end{enumerate}
 That is to say, the presence (resp. the absence) of a  viscous dissipation in the nonlinear Cattaneo-type model \eqref{Eq-Cattaneo-Fractional} is a criterion for the global in time existence (resp. the blow-up) of solutions. Note that the final conclusions will be shown by Table \ref{Table_3} in Section \ref{Sec-Final}.

To investigate the global (in time) behavior of solutions to the nonlinear viscous Cattaneo-type model \eqref{Eq-Cattaneo-Fractional} with $0<b\nu\ll 1$, we firstly study some qualitative properties of solutions to the corresponding linearized model, see the Cauchy problem in \eqref{Eq-Linear-Cattaneo-Fractional}. By applying the WKB analysis and the Fourier analysis to \eqref{Eq-Linear-Cattaneo-Fractional}, we derive optimal decay estimates in the $L^2$ framework, whose optimality is guaranteed by the same kind of upper and lower bounds for $t\gg1$. Furthermore, we obtain the long-time asymptotic profiles of the solutions.

Thanks to these preparatory results for the linearized model, we are able to construct time-weighted (depending on $\alpha$) Sobolev spaces with suitable regularities and to prove the global (in time) existence of small data solutions to \eqref{Eq-Cattaneo-Fractional} by using a standard Banach fixed point argument. To overcome some difficulties arising from the handling of higher order derivatives $\nabla\psi\cdot \nabla\psi_t$, we will set a suitable evolution space when $\alpha\in(1/2,1]$, with some loss of regularity in comparison with the corresponding linearized problem. Simultaneously, upper bound estimates of some Sobolev norms of the solutions and their derivatives will be derived as a byproduct. By managing the nonlinear terms properly (different from the one in \cite{Liu-Qin-Zhang=2022}), we arrive at optimal lower bound estimates and asymptotic profiles of the solutions when $|\mb{B}_0|\neq0$, where 
\begin{align}\label{B0}
	\mb{B}_0:=\int_{\mb{R}^n}\left(\psi_1(x)+\tau\psi_2(x)-\frac{\tau B}{2Ac_0^2}|\psi_1(x)|^2-\tau|\nabla\psi_0(x)|^2\right)\mathrm{d}x.
\end{align}
 These results reveal different profiles (depending on $\alpha$) of the global (in time) solutions to the viscous Cattaneo-type model \eqref{Eq-Cattaneo-Fractional} with $0<b\nu\ll 1$, e.g., the large-time profile for the solution itself is given by
	\begin{align*}
	\psi(t,x)\sim \begin{cases}
		\displaystyle{\ml{F}^{-1}_{\xi\to x}\left(\frac{1}{b\nu|\xi|^{2\alpha}}\mathrm{e}^{-\frac{1}{b\nu}c_0^2|\xi|^{2-2\alpha}t}\right)}|\mb{B}_0|&\mbox{when} \ \ \alpha\in[0,1/2),\vspace{0.15cm}\\
		\displaystyle{\ml{F}^{-1}_{\xi\to x}\left(\frac{\sin(\sqrt{c_0^2-\frac{1}{4}b^2\nu^2}|\xi|t)}{\sqrt{c_0^2-\frac{1}{4}b^2\nu^2}|\xi|}\mathrm{e}^{-\frac{b\nu}{2}|\xi|t}\right)}|\mb{B}_0|&\mbox{when} \ \ \alpha=1/2,\vspace{0.15cm}\\
		\displaystyle{\ml{F}^{-1}_{\xi\to x}\left(\frac{\sin(c_0|\xi|t)}{c_0|\xi|}\mathrm{e}^{-\frac{b\nu}{2}|\xi|^{2\alpha}t}\right)}|\mb{B}_0|&\mbox{when}\ \ \alpha\in(1/2,1].
	\end{cases}
\end{align*}
In other words, the global (in time) solutions in the viscous case have the anomalous diffusion (cf. \cite{BKW=2001,CC=2004,J=2005})  profiles when $\alpha\in[0,1/2)$ and the diffusion wave profiles when $\alpha\in[1/2,1]$ for large-time.

Let us turn to the inviscid case. By constructing three different time-dependent functionals with the the help of a function  introduced in \cite{Yordanov-Zhang=2006}, if the initial data fulfill some non-negative and compactly supported assumptions, we demonstrate the blow-up of the energy solutions to the nonlinear inviscid Cattaneo-type model \eqref{Eq-Cattaneo-Fractional} with $b\nu=0$ when $n=1,2,3$. Additionally, we will obtain upper bound estimates for the lifespan as a byproduct. Our approach is based on an iteration method and inspired by \cite{Chen-Palmieri=2021}.  Consequently, we will be able to understand the blow-up mechanism of the nonlinear Cattaneo-type model \eqref{Eq-Cattaneo-Fractional} generated by a zero viscosity. 

\subsection{Notations}
We define the following zones in the Fourier space:
\begin{align*}
	\ml{Z}_{\intt}(\varepsilon_0):=\{|\xi|\leqslant\varepsilon_0\ll1\},\ \ 
	\ml{Z}_{\bdd}(\varepsilon_0,N_0):=\{\varepsilon_0\leqslant |\xi|\leqslant N_0\},\ \ 
	\ml{Z}_{\extt}(N_0):=\{|\xi|\geqslant N_0\gg1\},
\end{align*}
 and we introduce the cut-off functions $\chi_{\intt}(\xi),\chi_{\bdd}(\xi),\chi_{\extt}(\xi)\in \mathcal{C}^{\infty}$ with supports in $\ml{Z}_{\intt}(\varepsilon_0)$, $\ml{Z}_{\bdd}(\varepsilon_0/2,2N_0)$ and $\ml{Z}_{\extt}(N_0)$, respectively, such that $\chi_{\bdd}(\xi)=1-\chi_{\extt}(\xi)-\chi_{\intt}(\xi)$ for all $\xi \in \mb{R}^n$.

The writing $f\lesssim g$  means that there exists a positive constant $C$ fulfilling $f\leqslant Cg$, which may be changed from line to line and, analogously, for $f\gtrsim g$. Furthermore, we write $f\simeq  g$  if and only if $f\lesssim g$ and $f\gtrsim g$, concurrently. $B_R$ denotes the ball around the origin with radius $R$ in $\mb{R}^n$.

Moreover, $\dot{H}^s_q$ with $s\in\mb{R}$ and $1\leqslant q<\infty$ denotes the Bessel potential space based on $L^q$, where $|D|^s$ with $s\in\mb{R}$ stands for the pseudo-differential operator with its symbol $|\xi|^s$. We recall the weighted $L^1$ space such that $L^{1,1}:=\{ f\in L^1:|x|f\in L^1\}$ and the notation $P_f:=\int_{\mb{R}^n}f(x)\mathrm{d}x$ for the zero momentum of $f$. 

\section{Main results}
After introducing some preparatory results in Section \ref{SubSec-2.1}, we state global (in time) behavior results for  the viscous Cattaneo-type model \eqref{Eq-Cattaneo-Fractional} with $0<b\nu\ll 1$ in the anomalous diffusion case $\alpha\in[0,1/2)$ and the diffusion wave case $\alpha\in[1/2,1]$, respectively, in Sections \ref{SubSec-2.2} and \ref{SubSec-2.3}. Finally, the blow-up result with the corresponding lifespan estimates for the inviscid Cattaneo-type model \eqref{Eq-Cattaneo-Fractional} when $b\nu=0$ is presented in Section \ref{Subsec-2.4}. 
\subsection{Auxiliary functions and explanations}\label{SubSec-2.1}
Let us introduce several Fourier multipliers as follows:
\begin{align}
\ml{G}_{1,j}(t,x)&:=\ml{F}^{-1}_{\xi\to x}\left(\chi_{\intt}(\xi)\frac{1}{b\nu}|\xi|^{-2\alpha}\partial_t^j\mathrm{e}^{-\frac{1}{b\nu}c_0^2|\xi|^{2-2\alpha}t}\right)\ \ \ \ \qquad\qquad\  \,\mbox{when} \ \ \alpha\in[0,1/2),\label{mlG1}\\
\ml{G}_{2,j}(t,x)&:=\ml{F}^{-1}_{\xi\to x}\left(\chi_{\intt}(\xi)\frac{\partial_t^j\sin(c_0|\xi|t)}{c_0|\xi|}\mathrm{e}^{-\frac{b\nu}{2}|\xi|^{2\alpha}t}\right)\ \ \ \ \ \qquad\qquad\  \,\mbox{when} \ \ \alpha\in(1/2,1],\label{mlG2}\\
\ml{G}_{3,j}(t,x)&:=\ml{F}^{-1}_{\xi\to x}\left(\chi_{\intt}(\xi)\frac{\partial_t^j\left(\sin(\frac{1}{2}\sqrt{4c_0^2-b^2\nu^2}|\xi|t)\mathrm{e}^{-\frac{b\nu}{2}|\xi|t}\right)}{\frac{1}{2}\sqrt{4c_0^2-b^2\nu^2}|\xi|}\right)\ \ \mbox{when} \ \ \alpha=1/2,\label{mlG3}
\end{align}
with $j=0,1,2$, which will appear in the asymptotic profiles of solutions. Recalling the definition of $|\mb{B}_0|$ in \eqref{B0}, let us give some explanations and comments for these profiles.
\begin{itemize}
	\item When $\alpha\in[0,1/2)$, the solutions $\partial_t^j\psi(t,\cdot)$ behave as the anomalous diffusion functions $\ml{G}_{1,j}(t,\cdot)\mb{B}_0$ in the $L^2$ framework, which consist of the singularity $\chi_{\intt}(\xi)|\xi|^{-2\alpha}$ and the anomalous diffusion multiplier $\exp(-c|\xi|^{2-2\alpha}t)$. For $\alpha=0$, the singularity disappears, and the multiplier becomes a Gaussian multiplier.
	\item When $\alpha\in[1/2,1]$, the solutions $\partial_t^j\psi(t,\cdot)$ behave as the diffusion wave functions $\ml{G}_{k,j}(t,\cdot)\mb{B}_0$ in the $L^2$ framework with $k=2,3$, which consist of the product of the singular factor $\chi_{\intt}(\xi)|\xi|^{-1}$, the oscillating multiplier $\sin(\omega_0 |\xi|t)$  as in the wave equation,  and the other anomalous diffusion multiplier $\exp(-c|\xi|^{2\alpha}t)$.
\end{itemize}

Let us discuss more in detail, the non-trivial quantity $\mb{B}_0$ introduced in \eqref{B0}, which  consists of two components 
\begin{align}
\mb{B}_0=P_{\psi_1+\tau\psi_2}-\tau P_{\frac{B}{2Ac_0^2}|\psi_1|^2+|\nabla \psi_0|^2}. \label{def B0 profile}
\end{align} The solution to the linearized Cauchy problem contributes to the first component $P_{\psi_1+\tau\psi_2}$. On the other hand, the solution to the nonlinear Cauchy problem with homogeneous initial data contributes to the second component 	$-\tau P_{\frac{B}{2Ac_0^2}|\psi_1|^2+|\nabla\psi_0|^2}$,
which originates from the nonlinear terms at $t=0$. 


We end this subsection introducing the following spaces for the Cauchy data:
\begin{align}\label{A1-SPACE}
	\ml{A}_{s,j}^{(1)}:=(H^{s+\max\{1,2-j\}}\cap L^1)\times(H^{s+1}\cap L^1)\times (H^s\cap L^1)
\end{align}
when $\alpha\in[0,1/2]$, and
\begin{align}\label{A2-SPACE}
	\ml{A}_{s,j}^{(2)}:=\begin{cases}
		({H}^{s+2}\cap L^1)\times ({H}^{s+2\alpha}\cap L^1)\times ({H}^s\cap L^1)&\mbox{when}\ \ j=0,\\
		({H}^{s+1}\cap L^1)\times({H}^{s+1}\cap L^1)\times({H}^{s+1-2\alpha}\cap L^1)&\mbox{when}\ \ j=1,\\
		({H}^{s+2-2\alpha}\cap L^1)\times ({H}^{s+2-2\alpha}\cap L^1)\times ({H}^s\cap L^1)&\mbox{when}\ \ j=2,
	\end{cases}
\end{align}
when $\alpha\in(1/2,1]$ with some parameters $s$ and $j=0,1,2$.

\subsection{Global (in time) behavior of the solutions in the anomalous diffusion case}\label{SubSec-2.2}
We now state the first main result in the present paper.
\begin{theorem}\label{Thm-01}
Let $s>n/2-1$ for all $n\geqslant 2$. Let us consider the nonlinear viscous Cattaneo-type model \eqref{Eq-Cattaneo-Fractional} with $0<b\nu\ll 1$ and $\alpha\in[0,1/2)$ for $n\geqslant 3$ and $\alpha\in(0,1/2)$ for $n=2$. There exists a constant $\epsilon>0$ such that for all $(\psi_0,\psi_1,\psi_2)\in\ml{A}_{s,0}^{(1)}$ with $\|(\psi_0,\psi_1,\psi_2)\|_{\ml{A}_{s,0}^{(1)}}\leqslant\epsilon$, there is a uniquely determined Sobolev solution
\begin{align*}
	\psi\in\ml{C}([0,\infty),H^{s+2})\cap\ml{C}^1([0,\infty),H^{s+1})\cap \ml{C}^2([0,\infty),H^s).
\end{align*} 
 For $j=0,1,2$ and, either $s_0=j-2$ or $s_0=s$, we have that the following estimates are fulfilled:
\begin{align}
\|\partial_t^j\psi(t,\cdot)\|_{\dot{H}^{s_0+2-j}}&\lesssim (1+t)^{-(j+1)-\frac{2s_0+n-2j}{2(2-2\alpha)}}\|(\psi_0,\psi_1,\psi_2)\|_{ \ml{A}_{s,0}^{(1)}},\label{Prof-02}\\
\left\|\left(\partial_t^j\psi-\ml{G}_{1,j}\mb{B}_0\right)(t,\cdot)\right\|_{\dot{H}^{s_0+2-j}}&=o\left(t^{-(j+1)-\frac{2s_0+n-2j}{2(2-2\alpha)}}\right),\label{Prof-01}
\end{align} where $\ml{G}_{1,j}$ is defined in \eqref{mlG1}.
 Furthermore, the following lower bound  estimates  hold:
\begin{align}
\|\partial_t^j\psi(t,\cdot)\|_{\dot{H}^{s_0+2-j}}\gtrsim t^{-(j+1)-\frac{2s_0+n-2j}{2(2-2\alpha)}}|\mb{B}_0| \label{Prof-03}
\end{align}
for $t\gg1$, provided that $|\mb{B}_0|\neq0$.
\end{theorem}
\begin{remark}
The upper and lower bounds for $\|\partial_t^j \psi (t,\cdot)\|_{\dot{H}^{s_0+2-j}}$ in \eqref{Prof-02} and \eqref{Prof-03} show that, for $j=0,1,2$ and, either $s_0=j-2$ or $s_0=s$, the optimal decay rates for long-time are given by 
\begin{align*}
	\|\partial_t^j\psi(t,\cdot)\|_{\dot{H}^{s_0+2-j}}\simeq t^{-(j+1)-\frac{2s_0+n-2j}{2(2-2\alpha)}}
\end{align*}
provided that $|\mb{B}_0|\neq0$. Moreover, the error estimates in \eqref{Prof-01} tell us that the anomalous diffusion functions $\ml{G}_{1,j}(t,\cdot)\mb{B}_0$ are the large-time profiles in the $L^2$ framework of $\partial_t^j \psi(t,\cdot)$ when $\alpha\in[0,1/2)$.
\end{remark}
\begin{remark}
As we will discuss in Remark \ref{Rem01}, by assuming additionally
\begin{align*}
\psi_1+\tau\psi_2\in L^{1,1}\ \ \mbox{and}\ \ \frac{B}{2Ac_0^2}|\psi_1|^2+|\nabla\psi_0|^2\in L^{1,1},	
\end{align*} the decay rates of the error estimates \eqref{B2} can be improved by $-\frac{\min\{\alpha,1-2\alpha\}}{1-\alpha}$ as $\alpha\in[0,1/2)$. Thus, by combining \eqref{B1}, \eqref{B2} with the improved decay rates, and \eqref{B3}, \eqref{C3}, we may derive
\begin{align*}
	&(1+t)^{j+1+\frac{2s_0+n-2j}{2(2-2\alpha)}}\left\|\left(\partial_t^j\psi-\ml{G}_{1,j}\mb{B}_0\right)(t,\cdot)\right\|_{\dot{H}^{s_0+2-j}}\\
	&\qquad\lesssim (1+t)^{-\frac{\min\{2\alpha,1-2\alpha\}}{2(1-\alpha)}}\left(\|(\psi_0,\psi_1,\psi_2)\|_{ \ml{A}_{s,0}^{(1)}}+\|\psi_1+\tau\psi_2\|_{L^{1,1}}+\left\|\frac{B}{2Ac_0^2}|\psi_1|^2+|\nabla\psi_0|^2\right\|_{L^{1,1}}\right)
\end{align*}
for any $t\geqslant0$ and, either $s_0=j-2$ or $s_0=s$ with $j=0,1,2$. The above estimates not only improve those in \eqref{Prof-01} with the additional decay rate $-\frac{\min\{2\alpha,1-2\alpha\}}{2(1-\alpha)}$, but also indicate the presence of a new threshold $\alpha=1/4$ for the second-order asymptotic profiles which is due to the different additional decay rates in the two cases $\alpha\in[0,1/4]$ and $\alpha\in(1/4,1/2)$. This is one of the new effect from the fractional Laplacians.
\end{remark}

\begin{remark} 
In the statement of Theorem \ref{Thm-01}, we excluded the case $\alpha=0$ when $n=2$ for technical reasons. We point out that in Remark \ref{Remark alpha=0 and n=2} we will provide a weaker global (in time) existence result in this case.
\end{remark}

\subsection{Global (in time) behavior of the solutions in the diffusion wave case}\label{SubSec-2.3}
The asymptotic profiles for solutions to \eqref{Eq-Cattaneo-Fractional} will be changed from the anomalous diffusion functions (when $\alpha\in[0,1/2)$ in Theorem \ref{Thm-01}) to the diffusion wave functions (when $\alpha\in[1/2,1]$ in the next theorems). This is one of the main effect of  the fractional Laplacian $(-\Delta)^{\alpha}$ in \eqref{Eq-Cattaneo-Fractional}.
\begin{theorem}\label{Thm-02}
	Let $s>n/2-1$ for all $n\geqslant 3$. Let us consider the nonlinear viscous Cattaneo-type model \eqref{Eq-Cattaneo-Fractional} with $0<b\nu\ll 1$ and $\alpha\in(1/2,1]$. There exists a constant $\epsilon>0$ such that for all $(\psi_0,\psi_1,\psi_2)\in\ml{A}_{s,0}^{(2)}$ with $\|(\psi_0,\psi_1,\psi_2)\|_{\ml{A}_{s,0}^{(2)}}\leqslant\epsilon$, there is a uniquely determined Sobolev solution
	\begin{align}\label{D1}
		\psi\in\ml{C}([0,\infty),H^{s+2})\cap\ml{C}^1([0,\infty),H^{s+1})\cap \ml{C}^2([0,\infty),H^{s+1-2\alpha}).
	\end{align} For $j=0,1$ and, either $s_0=j-2$ or $s_0=s$, and either $s_1=2\alpha-1$ or $s_1=s$, we have that the following estimates are fulfilled:
		\begin{align*}
	\|\partial_t^j\psi(t,\cdot)\|_{\dot{H}^{s_0+2-j}}&\lesssim (1+t)^{-\frac{2s_0+n+2}{4\alpha}}\|(\psi_0,\psi_1,\psi_2)\|_{ \ml{A}_{s,0}^{(2)}},\\
	\left\|\left(\partial_t^j\psi-\ml{G}_{2,j}\mb{B}_0\right)(t,\cdot)\right\|_{\dot{H}^{s_0+2-j}}&=o\left(t^{-\frac{2s_0+n+2}{4\alpha}}\right),
\end{align*}	
moreover,
\begin{align*}
		\|\partial_t^2\psi(t,\cdot)\|_{\dot{H}^{s_1+1-2\alpha}}&\lesssim(1+t)^{-\frac{2s_1+n+4-4\alpha}{4\alpha}}\|(\psi_0,\psi_1,\psi_2)\|_{ \ml{A}_{s,0}^{(2)}},\\
	\left\|\left(\partial_t^2\psi-\ml{G}_{2,2}\mb{B}_0\right)(t,\cdot)\right\|_{\dot{H}^{s_1+1-2\alpha}}&=o\left(t^{-\frac{2s_1+n+4-4\alpha}{4\alpha}}\right).
	\end{align*}
	  Furthermore, the following lower bound  estimates hold:
	\begin{align*}
	\|\partial_t^j\psi(t,\cdot)\|_{\dot{H}^{s_0+2-j}}&\gtrsim t^{-\frac{2s_0+n+2}{4\alpha}}|\mb{B}_0|,\\
	\|\partial_t^2\psi(t,\cdot)\|_{\dot{H}^{s_1+1-2\alpha}}&\gtrsim t^{-\frac{2s_1+n+4-4\alpha}{4\alpha}}|\mb{B}_0|,
	\end{align*}
	for $t\gg1$, provided that $|\mb{B}_0|\neq0$.
\end{theorem}

\begin{theorem}\label{Thm-03}
	Let $s>n/2-1$ for all $n\geqslant 2$. Let us consider the nonlinear viscous Cattaneo-type model \eqref{Eq-Cattaneo-Fractional} with $0<b\nu\ll 1$ and $\alpha=1/2$. There exists a constant $\epsilon>0$ such that for all $(\psi_0,\psi_1,\psi_2)\in\ml{A}_{s,0}^{(1)}$ with $\|(\psi_0,\psi_1,\psi_2)\|_{\ml{A}_{s,0}^{(1)}}\leqslant\epsilon$, there is a uniquely determined Sobolev solution
	\begin{align*}
		\psi\in\ml{C}([0,\infty),H^{s+2})\cap\ml{C}^1([0,\infty),H^{s+1})\cap\ml{C}^2([0,\infty),H^s).
	\end{align*} For $j=0,1,2$ and, either $s_0=j-2$ or $s_0=s$, we have that the following estimates are fulfilled:
		\begin{align*}
	\|\partial_t^j\psi(t,\cdot)\|_{\dot{H}^{s_0+2-j}}&\lesssim (1+t)^{-(s_0+1)-\frac{n}{2}}\|(\psi_0,\psi_1,\psi_2)\|_{ \ml{A}_{s,0}^{(1)}},\\
	\left\|\left(\partial_t^j\psi-\ml{G}_{3,j}\mb{B}_0\right)(t,\cdot)\right\|_{\dot{H}^{s_0+2-j}}&=o\left(t^{-(s_0+1)-\frac{n}{2}}\right).
	\end{align*}
	 Furthermore, the following lower bound  estimates hold:
	\begin{align*}
	\|\partial_t^j\psi(t,\cdot)\|_{\dot{H}^{s_0+2-j}}\gtrsim t^{-(s_0+1)-\frac{n}{2}}|\mb{B}_0|
	\end{align*}
	for $t\gg1$, provided that $|\mb{B}_0|\neq0$. 
\end{theorem}
\begin{remark}
Comparing Theorem \ref{Thm-03} with Theorems \ref{Thm-01} and \ref{Thm-02}, respectively, it is evident that
\begin{itemize}
	\item the regularities of the solutions and initial data in the threshold case $\alpha=1/2$ are the same as those for the anomalous diffusion case $\alpha\in[0,1/2)$. This can be understood by the analysis of the large frequencies in the linear problem;
	\item the decay rates and asymptotic profiles for the threshold case $\alpha=1/2$  coincide with those for the diffusion wave case as $\alpha\downarrow1/2$. This can be explained by the analysis of the small frequencies in the linear problem.
\end{itemize}
\end{remark}

\begin{remark} 
	The upper and lower bounds for the norm of $\psi(t,\cdot),\partial_t \psi(t,\cdot), \partial_t^2 \psi(t,\cdot)$  in suitable homogeneous Sobolev spaces in Theorems \ref{Thm-02} and \ref{Thm-03} show that, if $|\mb{B}_0|\neq0$,  the optimal decay rates for long-time are given by
	\begin{align*}
		\|\partial_t^j\psi(t,\cdot)\|_{\dot{H}^{s_0+2-j}}&\simeq t^{-\frac{2s_0+n+2}{4\alpha}}\ \ \ \quad \mbox{when}\ \ \alpha\in[1/2,1],\\
	\|\partial_t^2\psi(t,\cdot)\|_{\dot{H}^{s_1+1-2\alpha}}&\simeq t^{-\frac{2s_1+n+4-4\alpha}{4\alpha}}\ \ \, \mbox{when}\ \ \alpha\in(1/2,1],\\
	\|\partial_t^2\psi(t,\cdot)\|_{\dot{H}^{s_2}}&\simeq t^{-(s_2+1)-\frac{n}{2}}\ \ \,\,\,\,\,  \mbox{when}\ \ \alpha=1/2,
\end{align*}
	for either $s_0=j-2$ or $s_0=s$ with $j=0,1$; either $s_1=2\alpha-1$ or $s_1=s$; either $s_2=0$ or $s_2=s$. Moreover,  the diffusion wave functions $\ml{G}_{k,j}(t,\cdot)\mb{B}_0$ for $k=2$ and $k=3$, respectively, are the long-time profiles of $\partial_t^j \psi(t,\cdot)$ in the $L^2$ framework when $\alpha=1/2$ and when $\alpha\in(1/2,1]$, respectively.

\end{remark}
\subsection{Blow-up of the solutions to the inviscid Cattaneo-type model}\label{Subsec-2.4}
For the sake of clarity and readability, we rewrite the nonlinear inviscid Cattaneo-type model, i.e. the Cauchy problem \eqref{Eq-Cattaneo-Fractional} with $b\nu=0$, as follows:
\begin{align}\label{Eq-Cattaneo-inviscid}
	\begin{cases}
		(\tau\partial_t+\ml{I})\left(\psi_{tt}-c_0^2\Delta\psi\right)=\partial_t\ml{N}_{\psi},&x\in\mb{R}^n,\ t>0,\\
		\psi(0,x)=\epsilon_0\psi_0(x),\ \psi_t(0,x)=\epsilon_0\psi_1(x),\ \psi_{tt}(0,x)=\epsilon_0\psi_2(x),&x\in\mb{R}^n,
	\end{cases}
\end{align}
where $\epsilon_0$ is a positive parameter describing the size of initial data and it will be used later in the lifespan estimates. Before stating the blow-up result, let us introduce the class of weak solutions to the Cauchy problem \eqref{Eq-Cattaneo-inviscid} which will be considered in our blow-up result.
\begin{defn}\label{Defn-energy-solution}
Let $(\psi_0,\psi_1,\psi_2)\in H^2\times H^1\times L^2$. We say that $\psi=\psi(t,x)$ is an energy solution to the nonlinear inviscid Cattaneo-type model \eqref{Eq-Cattaneo-inviscid} on $[0,T)$ if
\begin{align*}
\psi\in\ml{C}([0,T ),H^2)\cap \ml{C}^1([0,T ),H^1)\cap \ml{C}^2([0,T ),L^2)\ \ \mbox{such that}\ \ \psi_t,\nabla\psi\in L^2_{\mathrm{loc}}([0,T )\times \mb{R}^n)
\end{align*}
fulfills 
 the integral relation
\begin{align}\label{P0}
	& \int_0^t\int_{\mb{R}^n} \psi(\eta,x) (-\tau \partial_t+\mathcal{I})(\partial_t^2-c_0^2 \Delta)\phi(\eta,x) \mathrm{d}x\mathrm{d}\eta  \notag\\
	&+ \int_{\mathbb{R}^n} \left[\tau \psi_{tt}(t,x) \phi(t,x)+\psi_t(t,x)\left(\phi(t,x)-\tau \phi_t(t,x)\right)+\psi(t,x)(\tau \phi_{tt}(t,x)-\phi_t(t,x)-\tau c_0^2\Delta \phi (t,x))\right] \mathrm{d}x\notag\\
	&=\epsilon_0\int_{\mathbb{R}^n} \left[\tau \psi_2(x) \phi(0,x)+\psi_1(x)\left(\phi(0,x)-\tau \phi_t(0,x)\right)+\psi_0(x)\left(\tau \phi_{tt}(0,x)-\phi_t(0,x)-\tau c_0^2\Delta \phi (0,x)\right)\right] \mathrm{d}x \notag\\
	& \quad -\epsilon_0^2\int_{\mb{R}^n}\left(\frac{B}{2Ac_0^2}|\psi_1(x)|^2+|\nabla\psi_0(x)|^2\right)\phi(0,x)\mathrm{d}x +\int_{\mb{R}^n}\ml{N}_{\psi}(t,x)\phi(t,x)\mathrm{d}x\notag\\
	&\quad -\int_0^t\int_{\mb{R}^n}\ml{N}_{\psi}(\eta,x)\phi_t(\eta,x)\mathrm{d}x\mathrm{d}\eta  
\end{align} 
for any $\phi\in\ml{C}_0^{\infty}([0,T )\times \mb{R}^n)$ and any $t\in(0,T )$.
\end{defn}
Let us recall the function $\Phi=\Phi(x)$, which was introduced for the first time in \cite{Yordanov-Zhang=2006}, via 
\begin{align}\label{BL-01}
	\Phi(x):=\begin{cases}
		\mathrm{e}^{x}+\mathrm{e}^{-x}&\mbox{when}\ \ n=1,\\
		\displaystyle{\int_{\mb{S}^{n-1}}\mathrm{e}^{x\cdot\omega}\mathrm{d}\sigma_{\omega}}&\mbox{when}\ \ n\geqslant 2.
	\end{cases}
\end{align}
Then, we have the next blow-up result and upper bound estimates for the lifespan.

\begin{theorem}\label{Thm-Blow-up} 
Let us consider the nonlinear inviscid Cattaneo-type model \eqref{Eq-Cattaneo-inviscid} with $\tau c_0 >1$. Let $(\psi_0,\psi_1,\psi_2)\in H^2\times H^1\times L^2$ be compactly supported functions with supports contained in $B_R$ for $R>0$ such that
\begin{align}
	&\int_{\mb{R}^n}\psi_1(x)\Phi(x)\mathrm{d}x\geqslant 0, \label{thm bu data 1} \\
& \int_{\mb{R}^n}\Big[ \tau\psi_2(x)+(\tau c_0+1)\psi_1(x)+c_0\psi_0(x)-\epsilon_0 \Big(\frac{B}{2A c_0^2}|\psi_1(x)|^2+|\nabla \psi_0(x)|^2\Big)\Big]\Phi(x)\mathrm{d}x>0, \label{thm bu data 2} \\
& \int_{\mb{R}^n}\Big[ \psi_2(x)- c_0^2\psi_0(x)-\frac{\epsilon_0}{\tau} \Big(\frac{B}{2A c_0^2}|\psi_1(x)|^2+|\nabla \psi_0(x)|^2\Big)\Big]\Phi(x)\mathrm{d}x\geqslant 0. \label{thm bu data 3}
\end{align}
Let $\psi$ be the energy solution to the Cauchy problem \eqref{Eq-Cattaneo-inviscid} according to Definition \ref{Defn-energy-solution} with the lifespan $T(\epsilon_0)$ satisfying
\begin{align*}
\mathrm{supp}\,\psi(t,\cdot)\subset B_{R+c_0t}\ \ \mbox{for any}\ \ t\in(0,T).
\end{align*}
Then, there exists a positive constant $\epsilon_1=\epsilon_1(\psi_0,\psi_1,\psi_2,n,R,\tau,c_0,A,B)$ such that for any $\epsilon_0\in(0,\epsilon_1]$ the solution $\psi$ blows up in finite time. Furthermore, the upper bound estimate for the lifespan
\begin{align}\label{Lifespan-Result}
	T(\epsilon_0)\leqslant\begin{cases}
	C\epsilon_0^{-\frac{2}{3-n}}&\mbox{when}\ \ n=1,2,\\
	\exp(C\epsilon_0^{-1})&\mbox{when}\ \ n=3,
	\end{cases}
\end{align}
holds with a positive constant $C$.
\end{theorem}
\begin{remark}
Comparing with the viscous case carrying $b\nu (-\Delta)^{\alpha}\psi_t$, due to the lack of viscous dissipations in the nonlinear acoustic model, the local (in time) energy solution cannot be extended globally. In other words, although the linear MGT equation in the critical case is marginally stable, the solution to the nonlinear JMGT equation in the critical case blows up in finite time. Thus, this blow-up result enhances the picture of studies for the JMGT equation (see  \cite{Kaltenbacher-Lasiecka-Marchand-2011,Kaltenbacher-Lasiecka-Pos-2012} etc).
\end{remark}

\section{The linearized Cattaneo-type model involving fractional Laplacians} \label{Section 3}
As preparations for studying global (in time) behavior of the Sobolev solutions to the nonlinear viscous Cattaneo-type model \eqref{Eq-Cattaneo-Fractional}, this section contributes to investigations of the corresponding linearized Cauchy problem to the viscous Cattaneo-type model \eqref{Eq-Cattaneo-Fractional} with the vanishing right hand side, namely,
\begin{align}\label{Eq-Linear-Cattaneo-Fractional}
	\begin{cases}
(\tau\partial_t+\ml{I})(\varphi_{tt}-c_0^2\Delta\varphi+b\nu(-\Delta)^{\alpha} \varphi_t)=0,&x\in\mb{R}^n,\ t>0,\\
		\varphi(0,x)=\varphi_0(x),\ \varphi_t(0,x)=\varphi_1(x),\ \varphi_{tt}(0,x)=\varphi_2(x),&x\in\mb{R}^n,
	\end{cases}
\end{align}
with $\alpha\in[0,1]$. Throughout this section, we will always assume $0<b\nu\ll 1$, which will be implicit in all statements of Section \ref{Section 3}. In what follows,  we will derive some optimal decay estimates and asymptotic profiles of the solutions to \eqref{Eq-Linear-Cattaneo-Fractional} for three cases $\alpha\in[0,1/2)$, $\alpha\in(1/2,1]$ and $\alpha=1/2$. Note that formally taking the singular limit case $\tau=0$, the third-order (in time) PDEs \eqref{Eq-Linear-Cattaneo-Fractional} will reduce to the  structurally damped waves (see \cite{D'Abbicco-Reissig=2014,D'Abbicco-Ebert=2014,Ikehata=2014,Ikehata-Takeda=2019} and references therein). Nevertheless, some new effect from the operator $\tau\partial_t+\ml{I}$ occur not only for initial data with higher regularity, but also in large-time asymptotics of the solutions with a new threshold $\alpha=1/3$ for the linearized problem (see Remark \ref{Rem01} later). As fundamental preliminaries for studying the nonlinear viscous problem in \eqref{Eq-Cattaneo-Fractional}, we will estimate $\partial_t^j\varphi(t,\cdot)$ in some Sobolev norms for $j=0,\dots,3$. 

\subsection{Refined estimates of the solutions in the Fourier space}\label{Subsection-Refine-Fourier}

Let us apply the partial Fourier transform with respect to spatial variable $x\in\mb{R}^n$ to the linearized model \eqref{Eq-Linear-Cattaneo-Fractional}. Thus, we get
\begin{align}\label{Fourier-Eq}
\begin{cases}
(\tau\partial_t+\ml{I})(\widehat{\varphi}_{tt}+b\nu|\xi|^{2\alpha} \widehat{\varphi}_t+c_0^2|\xi|^2\widehat{\varphi})=0,&\xi\in\mb{R}^n,\ t>0,\\
\widehat{\varphi}(0,\xi)=\widehat{\varphi}_0(\xi),\ \widehat{\varphi}_t(0,\xi)=\widehat{\varphi}_1(\xi),\ \widehat{\varphi}_{tt}(0,\xi)=\widehat{\varphi}_2(\xi),&\xi\in\mb{R}^n.
\end{cases}
\end{align}
The characteristic equation of \eqref{Fourier-Eq} is the cubic equation $(\tau\lambda+1)(\lambda^2+b\nu|\xi|^{2\alpha}\lambda+c_0^2|\xi|^2)=0$, whose three roots are 
\begin{align*}
	\lambda_1=-\frac{1}{\tau}\ \ \mbox{and}\ \ \lambda_{2,3}=\frac{1}{2}\left(-b\nu|\xi|^{2\alpha}\pm\sqrt{b^2\nu^2|\xi|^{4\alpha}-4c_0^2|\xi|^2}\,\right).
\end{align*}
To understand the influence of the parameter $\alpha$ on the asymptotic behavior of the solutions (i.e., on the roots $\lambda_{2,3}$), we split our discussion into the next four cases.
\begin{description}
	\item[Case 1:] When $\alpha\in[0,1/2)$ as $\xi\in\ml{Z}_{\intt}(\varepsilon_0)$, or $\alpha\in(1/2,1]$ as $\xi\in\ml{Z}_{\extt}(N_0)$, the asymptotics  are
	\begin{align*}
		\lambda_2=-\frac{1}{b\nu}c_0^2|\xi|^{2-2\alpha}+\ml{O}(|\xi|^{4-6\alpha})\ \ \mbox{and}\ \ \lambda_3=-b\nu|\xi|^{2\alpha}+\ml{O}(|\xi|^{2-2\alpha}).
	\end{align*}
\item[Case 2:] When $\alpha\in[0,1/2)$ as $\xi\in\ml{Z}_{\extt}(N_0)$, or $\alpha\in(1/2,1]$ as $\xi\in\ml{Z}_{\intt}(\varepsilon_0)$, the asymptotics are
\begin{align*}
	\lambda_{2,3}=\pm ic_0|\xi|-\frac{b\nu}{2}|\xi|^{2\alpha}+i\ml{O}(|\xi|^{4\alpha-1}),
\end{align*}
where the remainder $\ml{O}(|\xi|^{4\alpha-1})$ is a real-valued function.
\item[Case 3:] When $\alpha\in[0,1/2)\cup(1/2,1]$ as $\xi\in\ml{Z}_{\bdd}(\varepsilon_0,N_0)$, the roots fulfill $\Re\lambda_{2,3}<0$.
\item[Case 4:]  When $\alpha=1/2$, the explicit roots are
\begin{align}\label{B4}
	\lambda_{2,3}=\frac{1}{2}\left(-b\nu\pm  i\sqrt{4c_0^2-b^2\nu^2}\,\right)|\xi|.
\end{align}
\end{description}

From the previous asymptotics for  the characteristic roots, we next discuss them one by one.  
\medskip

\noindent\underline{\textbf{Estimates in Case 1:}} Due to pairwise distinct characteristic roots, the solution to the Cauchy problem \eqref{Fourier-Eq} is expressed via
\begin{align}\label{Solution-Formula}
\widehat{\varphi}=\widehat{K}_0\widehat{\varphi}_0+\widehat{K}_1\widehat{\varphi}_1+\widehat{K}_2\widehat{\varphi}_2
\end{align}
with the kernels in the Fourier space such that
\begin{align*}
\widehat{K}_0:=\sum\limits_{j=1,2,3}\frac{\mathrm{e}^{\lambda_jt}\prod\limits_{k\neq j}\lambda_k}{\prod\limits_{k\neq j}(\lambda_j-\lambda_k)},\ \ \widehat{K}_1:=-\sum\limits_{j=1,2,3}\frac{\mathrm{e}^{\lambda_jt}\sum\limits_{k\neq j}\lambda_k}{\prod\limits_{k\neq j}(\lambda_j-\lambda_k)},\ \ \widehat{K}_2:=\sum\limits_{j=1,2,3}\frac{\mathrm{e}^{\lambda_jt}}{\prod\limits_{k\neq j}(\lambda_j-\lambda_k)},
\end{align*}
where the index $k$ runs in $\{1,2,3\}$ in the  previous formulas. Let us focus on the subcase with small frequencies firstly. By some lengthy but straightforward computations, when $\alpha\in[0,1/2)$, we may obtain the pointwise estimates 
\begin{align*}
\chi_{\intt}(\xi)|\partial_t^j\widehat{K}_0|&\lesssim\chi_{\intt}(\xi)\left(|\xi|^{(2-2\alpha)j}\mathrm{e}^{-c|\xi|^{2-2\alpha}t}+|\xi|^{2\alpha j+2-4\alpha}\mathrm{e}^{-c|\xi|^{2\alpha}t}\right),\\
\chi_{\intt}(\xi)\left(|\partial_t^j\widehat{K}_1|+|\partial_t^j\widehat{K}_2|\right)&\lesssim\chi_{\intt}(\xi)\left(|\xi|^{(2-2\alpha)j-2\alpha}\mathrm{e}^{-c|\xi|^{2-2\alpha}t}+|\xi|^{(j-1)2\alpha}\mathrm{e}^{-c|\xi|^{2\alpha}t}\right),
\end{align*} 
and the refined estimates
\begin{align*}
&\chi_{\intt}(\xi)\left(\,\left|\partial_t^j\widehat{K}_1-\frac{1}{b\nu}|\xi|^{-2\alpha}\partial_t^j\mathrm{e}^{-\frac{1}{b\nu}c_0^2|\xi|^{2-2\alpha}t}\right|+\left|\partial_t^j\widehat{K}_2-\frac{\tau}{b\nu}|\xi|^{-2\alpha}\partial_t^j\mathrm{e}^{-\frac{1}{b\nu}c_0^2|\xi|^{2-2\alpha}t}\right|\,\right)\\
&\qquad\lesssim\chi_{\intt}(\xi)\left(|\xi|^{(2-2\alpha)j+\min\{0,2-6\alpha\}}\mathrm{e}^{-c|\xi|^{2-2\alpha}t}+|\xi|^{(j-1)2\alpha}\mathrm{e}^{-c|\xi|^{2\alpha}t}\right),
\end{align*} 
with $j=0,\dots,3$, where the competition comes from the error estimates of $\partial_t^j\widehat{K}_1$ as follows:
\begin{align*} 
&\chi_{\intt}(\xi)\left|\frac{\partial_t^j\mathrm{e}^{-\frac{1}{b\nu}c_0^2|\xi|^{2-2\alpha}t+\ml{O}(|\xi|^{4-6\alpha})t}\left(\frac{1}{\tau}+\ml{O}(|\xi|^{2\alpha})\right)}{\frac{b\nu}{\tau}|\xi|^{2\alpha}+\ml{O}(|\xi|^{2-2\alpha})}
-\frac{\partial_t^j\mathrm{e}^{-\frac{1}{b\nu}c_0^2|\xi|^{2-2\alpha}t}\frac{1}{\tau}}{\frac{b\nu}{\tau}|\xi|^{2\alpha}}
 \right|\\
 &\qquad\lesssim\chi_{\intt}(\xi)|\xi|^{(2-2\alpha)j}\mathrm{e}^{-c|\xi|^{2-2\alpha}t}\left(1+|\xi|^{2-6\alpha}+|\xi|^{-2\alpha}\ml{O}(|\xi|^{4-6\alpha})t\int_0^1\mathrm{e}^{\ml{O}(|\xi|^{4-6\alpha})t\tau}\mathrm{d}\tau\right)\\
 &\qquad\lesssim \chi_{\intt}(\xi)|\xi|^{(2-2\alpha)j+\min\{0,2-6\alpha\}}\mathrm{e}^{-c|\xi|^{2-2\alpha}t}.
\end{align*}
The combinations of previous estimates and the representation \eqref{Solution-Formula} lead to
\begin{align}\label{Est-01}
	\chi_{\intt}(\xi)|\partial_t^j\widehat{\varphi}|\lesssim\chi_{\intt}(\xi)\left(|\xi|^{(2-2\alpha)j-2\alpha}\mathrm{e}^{-c|\xi|^{2-2\alpha}t}+|\xi|^{(j-1)2\alpha}\mathrm{e}^{-c|\xi|^{2\alpha} t}\right)\left(|\widehat{\varphi}_0|+|\widehat{\varphi}_1|+|\widehat{\varphi}_2|\right)
\end{align}
and
\begin{align*}
&\chi_{\intt}(\xi)\left|\partial_t^j\widehat{\varphi}-\frac{1}{b\nu}|\xi|^{-2\alpha}\partial_t^j\mathrm{e}^{-\frac{1}{b\nu}c_0^2|\xi|^{2-2\alpha}t}\left(\widehat{\varphi}_1+\tau\widehat{\varphi}_2\right)\right|\\
&\qquad\lesssim\chi_{\intt}(\xi)\left(|\xi|^{(2-2\alpha)j+\min\{0,2-6\alpha\}}\mathrm{e}^{-c|\xi|^{2-2\alpha}t}+|\xi|^{(j-1)2\alpha}\mathrm{e}^{-c|\xi|^{2\alpha}t}\right)\left(|\widehat{\varphi}_0|+|\widehat{\varphi}_1|+|\widehat{\varphi}_2|\right)
\end{align*}
with $j=0,\dots,3$. For the other subcase $\alpha\in(1/2,1]$ with large frequencies, we may derive
\begin{align}
	\chi_{\extt}(\xi)|\widehat{\varphi}|&\lesssim\chi_{\extt}(\xi)\mathrm{e}^{-ct}\left(|\widehat{\varphi}_0|+|\xi|^{2\alpha-2}|\widehat{\varphi}_1|+|\xi|^{-2}|\widehat{\varphi}_2|\right),\notag\\
	\chi_{\extt}(\xi)|\partial_t\widehat{\varphi}|&\lesssim\chi_{\extt}(\xi)\mathrm{e}^{-ct}\left(|\widehat{\varphi}_0|+|\widehat{\varphi}_1|+|\xi|^{-2\alpha}|\widehat{\varphi}_2|\right),\label{A3}\\
	\chi_{\extt}(\xi)|\partial_t^2\widehat{\varphi}|&\lesssim\chi_{\extt}(\xi)\mathrm{e}^{-ct}\left(|\xi|^{2-2\alpha}|\widehat{\varphi}_0|+|\xi|^{2-2\alpha}|\widehat{\varphi}_1|+|\widehat{\varphi}_2|\right),\notag\\
	\chi_{\extt}(\xi)|\partial_t^3\widehat{\varphi}|&\lesssim\chi_{\extt}(\xi)\mathrm{e}^{-ct}\left(|\xi|^2|\widehat{\varphi}_0|+|\xi|^2|\widehat{\varphi}_1|+|\xi|^{2\alpha}|\widehat{\varphi}_2|\right).\notag
\end{align}

\medskip
\noindent\underline{\textbf{Estimates in Case 2:}} Since $\lambda_{2,3}$ are complex conjugate roots, they can be re-expressed as $\lambda_{2,3}=\lambda_{\mathrm{R}}\pm i\lambda_{\mathrm{I}}$ with
\begin{align}\label{A1}
	\lambda_{\mathrm{R}}=-\frac{b\nu}{2}|\xi|^{2\alpha}\ \ \mbox{and}\ \ \lambda_{\mathrm{I}}=c_0|\xi|+\ml{O}(|\xi|^{4\alpha-1}).
\end{align}
Plugging the last expressions into \eqref{Solution-Formula}, we have
%
\begin{align}
\widehat{\varphi}&=\frac{(\lambda_{\mathrm{I}}^2+\lambda_{\mathrm{R}}^2)\widehat{\varphi}_0-2\lambda_{\mathrm{R}}\widehat{\varphi}_1+\widehat{\varphi}_2}{(\lambda_1-\lambda_{\mathrm{R}})^2+\lambda_{\mathrm{I}}^2}\mathrm{e}^{\lambda_1t}+\frac{(\lambda_1-2\lambda_{\mathrm{R}})\lambda_1\widehat{\varphi}_0+2\lambda_{\mathrm{R}}\widehat{\varphi}_1-\widehat{\varphi}_2}{(\lambda_1-\lambda_{\mathrm{R}})^2+\lambda_{\mathrm{I}}^2}\cos(\lambda_{\mathrm{I}}t)\mathrm{e}^{\lambda_{\mathrm{R}}t}\notag\\
&\quad+\frac{\lambda_1\left(\lambda_{\mathrm{R}}(\lambda_{\mathrm{R}}-\lambda_1)-\lambda_{\mathrm{I}}^2\right)\widehat{\varphi}_0+(\lambda_1^2-\lambda_{\mathrm{R}}^2+\lambda_{\mathrm{I}}^2)\widehat{\varphi}_1+(\lambda_{\mathrm{R}}-\lambda_1)\widehat{\varphi}_2}{\lambda_{\mathrm{I}}((\lambda_1-\lambda_{\mathrm{R}})^2+\lambda_{\mathrm{I}}^2)}\sin(\lambda_{\mathrm{I}}t)\mathrm{e}^{\lambda_{\mathrm{R}}t}.\label{Solution-Formula-2}
\end{align} 
Applying \eqref{A1} for $\xi\in\ml{Z}_{\intt}(\varepsilon_0)$ and $\alpha\in(1/2,1]$, one may obtain
\begin{align*}
\chi_{\intt}(\xi)|\partial_t^j\widehat{\varphi}|\lesssim\chi_{\intt}(\xi)|\xi|^{j}\left(1+\frac{|\sin(c_0|\xi|t)|}{c_0|\xi|}\right)\mathrm{e}^{-c|\xi|^{2\alpha}t}\left(|\widehat{\varphi}_0|+|\widehat{\varphi}_1|+|\widehat{\varphi}_2|\right)
\end{align*} 
and the refined estimates
\begin{align}\label{Est-03} 
\chi_{\intt}(\xi)\left|\partial_t^j\widehat{\varphi}-\frac{\partial_t^j\sin(c_0|\xi|t)}{c_0|\xi|}\mathrm{e}^{-\frac{b\nu}{2}|\xi|^{2\alpha}t}\left(\widehat{\varphi}_1+\tau\widehat{\varphi}_2\right) \right|&\lesssim\chi_{\intt}(\xi)|\xi|^{2\alpha-2+j}\mathrm{e}^{-c|\xi|^{2\alpha}t}\left(|\widehat{\varphi}_0|+|\widehat{\varphi}_1|+|\widehat{\varphi}_2|\right)
\end{align} 
with $j=0,\dots,3$, in which we used
\begin{align*}
\chi_{\intt}(\xi)\left|\frac{-\lambda_1^2}{\lambda_{\mathrm{I}}(2\lambda_{\mathrm{R}}\lambda_1-\lambda_{\mathrm{I}}^2-\lambda_{\mathrm{R}}^2-\lambda_1^2)}-\frac{1}{c_0|\xi|}\right|&\lesssim \chi_{\intt}(\xi)|\xi|^{4\alpha-3},\\
\chi_{\intt}(\xi)|\sin(\lambda_{\mathrm{I}}t)-\sin(c_0|\xi|t)|&\lesssim\chi_{\intt}(\xi)|\xi|^{4\alpha-1}t,
\end{align*}
in the error estimates. For $\xi\in\ml{Z}_{\extt}(N_0)$ and $\alpha\in[0,1/2)$, with $j=0,\dots,3$, we have
\begin{align}\label{Est-large}
	\chi_{\extt}(\xi)|\partial_t^j\widehat{\varphi}|&\lesssim\chi_{\extt}(\xi)\mathrm{e}^{-ct}\left(|\xi|^{\max\{j-1,0\}}|\widehat{\varphi}_0|+|\xi|^{j-1}|\widehat{\varphi}_1|+|\xi|^{j-2}|\widehat{\varphi}_2|\right).
\end{align}

\medskip
\noindent\underline{\textbf{Estimates in Case 3:}}
	 Since $\xi\not\in \ml{Z}_{\intt}(\varepsilon_0)\cup \ml{Z}_{\extt}(N_0)$, we have exponential decay
\begin{align}\label{Est-02}
	\chi_{\bdd}(\xi)|\partial_t^j\widehat{\varphi}|\lesssim\chi_{\bdd}(\xi)\mathrm{e}^{-ct}\left(|\widehat{\varphi}_0|+|\widehat{\varphi}_1|+|\widehat{\varphi}_2|\right)
\end{align}
with $j=0,\dots,3$.

\medskip
\noindent\underline{\textbf{Estimates in Case 4:}} By plugging the expressions
\begin{align*}
	\lambda_{\mathrm{R}}=-\frac{b\nu}{2}|\xi|\ \ \mbox{and}\ \ \lambda_{\mathrm{I}}=\frac{1}{2}\sqrt{4c_0^2-b^2\nu^2}|\xi|
\end{align*}
into the representation \eqref{Solution-Formula-2}, we consequently find
\begin{align*}
\chi_{\intt}(\xi)|\partial_t^j\widehat{\varphi}|\lesssim\chi_{\intt}(\xi)|\xi|^j\left(1+\frac{|\sin(\frac{1}{2}\sqrt{4c_0^2-b^2\nu^2}|\xi|t)|}{\frac{1}{2}\sqrt{4c_0^2-b^2\nu^2}|\xi|}\right)\mathrm{e}^{-c|\xi|t}\left(|\widehat{\varphi}_0|+|\widehat{\varphi}_1|+|\widehat{\varphi}_2|\right)
\end{align*}
and the error estimates
\begin{align*}
\chi_{\intt}(\xi)\left|\partial_t^j\widehat{\varphi}-\frac{\partial_t^j\left(\sin(\frac{1}{2}\sqrt{4c_0^2-b^2\nu^2}|\xi|t)\mathrm{e}^{-\frac{b\nu}{2}|\xi|t}\right)}{\frac{1}{2}\sqrt{4c_0^2-b^2\nu^2}|\xi|}\left(\widehat{\varphi}_1+\tau\widehat{\varphi}_2\right) \right|\lesssim\chi_{\intt}(\xi)|\xi|^j\mathrm{e}^{-c|\xi|t}\left(|\widehat{\varphi}_0|+|\widehat{\varphi}_1|+|\widehat{\varphi}_2|\right)
\end{align*}
with $j=0,\dots,3$. Additionally, the solution for $\alpha=1/2$ and for large frequencies fulfills \eqref{Est-large} as well.

\subsection{Decay properties and asymptotic profiles of the solutions}


Since we cannot apply the classical theory for operators of Kovalevskian-type or of Petrovskian-type to \eqref{Eq-Linear-Cattaneo-Fractional}, it is not obvious how the regularities for the Cauchy data and for the solution to \eqref{Eq-Linear-Cattaneo-Fractional} are related. Following the approach employed to prove the existence of Sobolev solutions to the classical free wave equation (e.g. \cite[Chapter 14]{Ebert-Reissig-book}), since the regularities for $\varphi, \partial_t \varphi,\partial_t^2 \varphi$ are provided by the estimates for large frequencies (that is, \eqref{A3} and \eqref{Est-large} for $\alpha\in (1/2,1]$ and for $\alpha\in [0,1/2]$, respectively),  one can straightforwardly show the next result.

\begin{prop}\label{PROP-EXIS}
Let $\alpha\in[0,1]$ and $n\geqslant 2$. Let us consider the linearized viscous problem \eqref{Eq-Linear-Cattaneo-Fractional} with $(\varphi_0,\varphi_1,\varphi_2)\in H^{s+2}\times H^{s+\max\{1,2\alpha\}}\times H^s$. Then, there is a uniquely determined Sobolev solution
\begin{align}\label{D2}
	\varphi\in\ml{C}([0,\infty),H^{s+2})\cap \ml{C}^1([0,\infty),H^{s+\max\{1,2\alpha\}})\cap \ml{C}^2([0,\infty),H^s)
\end{align}
with any $s\geqslant0$.
\end{prop}

Let us recall some decay estimates for two kinds of Fourier multipliers in the $L^2$ space (see \cite{Ikehata=2014,Ikehata-Takeda=2019}).  
\begin{lemma}\label{Lemma-01}
Let $\beta>0$ and $2s+n>0$. The following two sided estimate holds:
\begin{align*}
\ml{R}_0(t):=\left\|\chi_{\intt}(\xi)\widehat{g}_0(t,|\xi|)|\xi|^s\mathrm{e}^{-c|\xi|^{\beta}t}\right\|_{L^2}\simeq t^{-\frac{2s+n}{2\beta}}
\end{align*}
for $t\gg1$, where either for $\widehat{g}_0(t,|\xi|)=1$ or for $\widehat{g}_0(t,|\xi|)=\sin(C|\xi|t)$.   Moreover,  $\ml{R}_0(t)\lesssim 1$ holds for $t\leqslant 1$.
\end{lemma}

Now, we derive some $(L^2\cap L^1)-L^2$ and $L^2-L^2$ estimates for the solutions and their time-derivatives, considering three different subcases for the range of  $\alpha$.
\medskip

\noindent\underline{\bf{1st case: $\alpha\in[0,1/2)$}}

\begin{prop}\label{PROP-01}Let $\alpha\in[0,1/2)$ and $n\geqslant 2$. Let us consider the linearized viscous problem \eqref{Eq-Linear-Cattaneo-Fractional} with $(\varphi_0,\varphi_1,\varphi_2)\in \ml{A}_{s,j}^{(1)}$. Then, $\partial_t^j \varphi$ fulfills the upper bound  estimates
	 \begin{align}
	 \|\partial_t^j\varphi(t,\cdot)\|_{\dot{H}^{s+2-j}}&\lesssim (1+t)^{-(j+1)-\frac{2s+n-2j}{2(2-2\alpha)}}\|(\varphi_0,\varphi_1,\varphi_2)\|_{ \ml{A}_{s,j}^{(1)}},\label{A2}\\
\left\|\left(\partial_t^j\varphi-\ml{G}_{1,j}P_{\varphi_1+\tau\varphi_2}\right)(t,\cdot)\right\|_{\dot{H}^{s+2-j}}&=o\left(t^{-(j+1)-\frac{2s+n-2j}{2(2-2\alpha)}}\right),\notag
\end{align}
where the profiles $\ml{G}_{1,j}$ and the spaces $\ml{A}_{s,j}^{(1)}$ were introduced in \eqref{mlG1} and \eqref{A1-SPACE}, respectively. Moreover, the upper bound estimates in \eqref{A2} are optimal since the following lower bound estimates hold:
\begin{align*}
\|\partial_t^j\varphi(t,\cdot)\|_{\dot{H}^{s+2-j}}\gtrsim t^{-(j+1)-\frac{2s+n-2j}{2(2-2\alpha)}}|P_{\varphi_1+\tau\varphi_2}|
\end{align*}
for $t\gg1$, provided that $|P_{\varphi_1+\tau\varphi_2}|\neq0$. Here, we take either $s=j-2$ or $s\geqslant0$ with $j=0,1,2$.
\end{prop}
\begin{proof} Let $s=j-2$ or $s\geqslant 0$ with $j=0,1,2$.
By  the Hausdorff-Young inequality and Lemma \ref{Lemma-01}, we may have
\begin{align*}
\|\partial_t^j\varphi(t,\cdot)\|_{\dot{H}^{s+2-j}}&\lesssim \|\chi_{\intt}(\xi)|\xi|^{s+2-j}\partial_t^j\widehat{\varphi}(t,\xi)\|_{L^2}+\|(1-\chi_{\intt}(\xi))|\xi|^{s+2-j}\partial_t^j\widehat{\varphi}(t,\xi)\|_{L^2}\\
&\lesssim\left\|\chi_{\intt}(\xi)|\xi|^{(2-2\alpha)j-2\alpha+s+2-j}\mathrm{e}^{-c|\xi|^{2-2\alpha}t}\right\|_{L^2}\|(\widehat{\varphi}_0,\widehat{\varphi}_1,\widehat{\varphi}_2)\|_{(L^{\infty})^3}\\
&\quad+\mathrm{e}^{-ct}\|(1-\chi_{\intt}(D))|D|^{s+2-j}(\varphi_0,\varphi_1,\varphi_2)\|_{\dot{H}^{\max\{j-1,0\}}\times \dot{H}^{j-1}\times \dot{H}^{j-2}}\\
&\lesssim (1+t)^{-(j+1)-\frac{2s+n-2j}{2(2-2\alpha)}}\|(\varphi_0,\varphi_1,\varphi_2)\|_{\ml{A}_{s,j}^{(1)}},
\end{align*}
where we used \eqref{Est-01}, \eqref{Est-02} and \eqref{Est-large} in the second inequality and Lemma \ref{Lemma-01} in the third inequality. In particular, in this step we have to require the restriction $n\geqslant 2>4\alpha$ for the spatial dimension when $s=j-2$ and $j=0$. Similarly, we arrive at the error estimates
\begin{align}\label{err-01}
%
\left\|\partial_t^j\varphi (t,\cdot)-\ml{G}_{1,j}(t,|D|)(\varphi_1+\tau\varphi_2)\right\|_{\dot{H}^{s+2-j}} \color{black}\lesssim (1+t)^{-(j+1)-\frac{2s+n-2j}{2(2-2\alpha)}-\frac{\min\{\alpha,1-2\alpha\}}{1-\alpha}}\|(\varphi_0,\varphi_1,\varphi_2)\|_{\ml{A}_{s,j}^{(1)}}.
\end{align}
Considering the Fourier multiplier with the symbol $\widehat{g}_1(t,\xi):=\ml{F}_{x\to\xi}(g_1(t,x))$ acting on the function $f_0=f_0(x)$, by using the Lagrange theorem
\begin{align*}
	|g_1(t,x-y)-g_1(t,x)|\lesssim|y|\,|\nabla g_1(t,x-\theta_1y)|
\end{align*}
with $\theta_1\in(0,1)$, one notices
\begin{align}
\|g_1(t,D)f_0(\cdot)-g_1(t,\cdot)P_{f_0}\|_{L^2}&\lesssim\left\|\int_{|y|\leqslant t^{\beta_1}}\big(g_1(t,\cdot-y)-g_1(t,\cdot)\big)f_0(y)\mathrm{d}y\right\|_{L^2}\notag\\
&\quad+\left\|\int_{|y|\geqslant t^{\beta_1}}\big(|g_1(t,\cdot-y)|+|g_1(t,\cdot)|\big)|f_0(y)|\mathrm{d}y\right\|_{L^2}\notag\\
&\lesssim t^{\beta_1}\|\nabla g_1(t,\cdot)\|_{L^2}\|f_0\|_{L^1}+\|g_1(t,\cdot)\|_{L^2}\|f_0\|_{L^1(|x|\geqslant t^{\beta_1})}\label{Est-04}.
\end{align}
By taking $\widehat{g}_1(t,\xi)=\widehat{\ml{G}}_{1,j}(t,|\xi|)$, and
\begin{align*}
	f_0=\varphi_1+\tau\varphi_2\in L^1\ \ \mbox{so that}\ \ \|\varphi_1+\tau\varphi_2\|_{L^1(|x|\geqslant t^{\beta_1})}=o(1)  \ \ \mbox{as} \ \ t\to+\infty,
\end{align*} we are able to deduce
\begin{align}\label{err-02}
	&\|\ml{G}_{1,j}(t,|D|)(\varphi_1+\tau\varphi_2)(\cdot)-\ml{G}_{1,j}(t,\cdot)P_{\varphi_1+\tau\varphi_2}\|_{\dot{H}^{s+2-j}}\notag\\
	&\qquad\lesssim t^{\beta_1}\left\|\chi_{\intt}(\xi)|\xi|^{3-2\alpha+s-j}\partial_t^j\mathrm{e}^{-c|\xi|^{2-2\alpha}t}\right\|_{L^2}\|\varphi_1+\tau\varphi_2\|_{L^1}\notag\\
	&\qquad\quad+\left\|\chi_{\intt}(\xi)|\xi|^{2-2\alpha+s-j}\partial_t^j\mathrm{e}^{-c|\xi|^{2-2\alpha}t}\right\|_{L^2}\|\varphi_1+\tau\varphi_2\|_{L^1(|x|\geqslant t^{\beta_1})}\notag\\
	&\qquad=o\left(t^{-(j+1)-\frac{2s+n-2j}{2(2-2\alpha)}}\right)
\end{align}
as $t\gg1$ with $j=0,1,2$, where we chose $\beta_1$ to be a small constant. Combining \eqref{err-01} and \eqref{err-02} it results
\begin{align*}
	\left\|\left(\partial_t^j\varphi-\ml{G}_{1,j}P_{\varphi_1+\tau\varphi_2}\right)(t,\cdot)\right\|_{\dot{H}^{s+2-j}}=o\left(t^{-(j+1)-\frac{2s+n-2j}{2(2-2\alpha)}}\right).
\end{align*}
 For $t\gg1$, the Minkowski inequality implies that
\begin{align*}
\|\partial_t^j\varphi(t,\cdot)\|_{\dot{H}^{s+2-j}}&\gtrsim \|\ml{G}_{1,j}(t,\cdot)\|_{\dot{H}^{s+2-j}}|P_{\varphi_1+\tau\varphi_2}|-\left\|\left(\partial_t^j\varphi-\ml{G}_{1,j}P_{\varphi_1+\tau\varphi_2}\right)(t,\cdot)\right\|_{\dot{H}^{s+2-j}}\\
&\gtrsim t^{-(j+1)-\frac{2s+n-2j}{2(2-2\alpha)}}|P_{\varphi_1+\tau\varphi_2}|- o\left(t^{-(j+1)-\frac{2s+n-2j}{2(2-2\alpha)}}\right)\\
&\gtrsim t^{-(j+1)-\frac{2s+n-2j}{2(2-2\alpha)}}|P_{\varphi_1+\tau\varphi_2}|,
\end{align*}
 where we employed the lower bound  estimates in Lemma \ref{Lemma-01} to control $\|\,|D|^{s+2-j}\ml{G}_{1,j}(t,\cdot)\|_{L^2}$ from the below.
\end{proof}
\begin{remark}\label{Rem01}
In the same setting of Proposition \ref{PROP-01},	if we assume additionally $\varphi_1+\tau\varphi_2\in L^{1,1}$, from \cite[Lemma 2.2]{Ikehata=2014}, we get
	\begin{align*}
		\left|\widehat{\varphi_1+\tau\varphi_2}-P_{\varphi_1+\tau\varphi_2}\right|\lesssim|\xi|\,\|\varphi_1+\tau\varphi_2\|_{L^{1,1}},
	\end{align*}
	which helps us to improve \eqref{err-02} by
	\begin{align*}
		\|\ml{G}_{1,j}(t,|D|)(\varphi_1+\tau\varphi_2)(\cdot)-\ml{G}_{1,j}(t,\cdot)P_{\varphi_1+\tau\varphi_2}\|_{\dot{H}^{s+2-j}}\lesssim(1+t)^{-(j+1)-\frac{2s+n-2j}{2(2-2\alpha)}-\frac{1}{2-2\alpha}}\|\varphi_1+\tau\varphi_2\|_{L^{1,1}}.
	\end{align*}
	Combining these last estimates with  \eqref{err-01}, we immediately arrive at
	\begin{align}\label{B1}
		&\left\|\left(\partial_t^j\varphi-\ml{G}_{1,j}P_{\varphi_1+\tau\varphi_2}\right)(t,\cdot)\right\|_{\dot{H}^{s+2-j}}\notag\\
		&\qquad\lesssim(1+t)^{-(j+1)-\frac{2s+n-2j}{2(2-2\alpha)}-\frac{\min\{\alpha,1-2\alpha\}}{1-\alpha}}\left(\|(\varphi_0,\varphi_1,\varphi_2)\|_{\ml{A}_{s,j}^{(1)}}+\|\varphi_1+\tau\varphi_2\|_{L^{1,1}}\right)
	\end{align}
	for either $s=j-2$ or $s\geqslant0$ with $j=0,1,2$. Comparing with the estimates in \eqref{B1} with those in \eqref{A2},  we see that by subtracting the profiles $\ml{G}_{1,j}P_{\varphi_1+\tau\varphi_2}$ to $\partial_t^j \varphi$ the decay rates in \eqref{A2} have been improved with an additional decay of magnitude $-\frac{\min\{\alpha,1-2\alpha\}}{1-\alpha}$ when $\alpha\in[0,1/2)$. This is a new effect, which does not occur  in the singular limit case $\tau=0$, and that provides a new threshold $\alpha=1/3$ for the linearized viscous problem \eqref{Eq-Linear-Cattaneo-Fractional}. This phenomenon, a different improvement in the decay rates in the sub-ranges $\alpha\in [0,1/3)$ and $\alpha\in (1/3,1/2)$, suggests the presence of two different long-time second-order profiles for $\partial_t^j \varphi$ depending on the value of $\alpha$.
\end{remark}
\begin{coro}\label{CORO-01}Let $\alpha\in[0,1/2)$ and $n\geqslant 2$. Let us consider the linearized viscous problem \eqref{Eq-Linear-Cattaneo-Fractional} with $\varphi_0=\varphi_1=0$ and $\varphi_2\in\dot{H}^s$. Then,  the functions $\partial_t^j\varphi$ fulfill the upper bound  estimates
	\begin{align*}
	\|\partial_t^j\varphi(t,\cdot)\|_{\dot{H}^{s+2-j}}&\lesssim (1+t)^{-(j+1)+\frac{j}{2-2\alpha}}\|\varphi_2\|_{\dot{H}^s},
	\end{align*}
	where we take either $s=j-2$ or $s\geqslant0$ with $j=0,1,2$. Additionally,  we have the $L^2-L^2$ estimates
\begin{align}\label{extra L2 L2 estimate alpha< 1/2}
\| \partial_t^j \varphi (t,\cdot)\|_{L^2}\lesssim (1+t)^{1-j}\|\varphi_2\|_{L^2}
\end{align}
for $j=1,2$.	 Finally,	
 assuming $\varphi_2\in\dot{H}^{s+1}\cap L^1$, the third-order derivative of the solution fulfills
	\begin{align}
		\|\partial_t^3\varphi(t,\cdot)\|_{\dot{H}^s}&\lesssim (1+t)^{-3-\frac{2s+n-4\alpha}{2(2-2\alpha)}}\|\varphi_2\|_{\dot{H}^{s+1}\cap L^1},\label{C1}\\
		\|\partial_t^3\varphi(t,\cdot)\|_{\dot{H}^s}&\lesssim(1+t)^{-2-\frac{1-4\alpha}{2-2\alpha}}\|\varphi_2\|_{\dot{H}^{s+1}},\label{C2} 
	\end{align}
for $s\geqslant0$.
\end{coro}
\begin{proof}
Concerning the first estimates, we just need to consider the part of small frequencies, since as we showed in Proposition \ref{PROP-01}, for large frequencies we have exponential decay. From \eqref{Est-01}, we get
\begin{align*}
\left\|\chi_{\intt}(\xi)|\xi|^{s+2-j}\partial_t^j\widehat{\varphi}(t,\xi)\right\|_{L^2}&\lesssim\sup\limits_{|\xi|\leqslant\varepsilon_0}\left(|\xi|^{2-j+(2-2\alpha)j-2\alpha}\mathrm{e}^{-c|\xi|^{2-2\alpha}t}\right)\|\,|\xi|^s\widehat{\varphi}_2\|_{L^2}\\
&\lesssim (1+t)^{-(j+1)+\frac{j}{2-2\alpha}}\|\varphi_2\|_{\dot{H}^s},
\end{align*} 
in which $s=j-2$ or $s\geqslant0$ with $j=0,1,2$.
For the second estimates by applying \eqref{Est-01} and \eqref{Est-large}, we have
\begin{align*}
\| \partial_t^j\widehat{\varphi}(t,\xi)\|_{L^2} & \lesssim \left\|\chi_{\intt}(\xi)|\xi|^{(2-2\alpha)(j-1)+2(1-2\alpha)}\mathrm{e}^{-c|\xi|^{2-2\alpha}t}\right\|_{L^\infty}\|\widehat{\varphi}_2\|_{L^{2}}+\mathrm{e}^{-ct}\|(1-\chi_{\intt}(\xi))\widehat{\varphi}_2\|_{L^2} \\ &   \lesssim (1+t)^{1-j}\|\varphi_2\|_{L^2}.
\end{align*}
 For the third and fourth estimates, we use \eqref{Est-01} as well as \eqref{Est-large} to arrive at
\begin{align*}
\|\,|\xi|^s\partial_t^3\widehat{\varphi}(t,\xi)\|_{L^2}&\lesssim\left\|\chi_{\intt}(\xi)|\xi|^{3(2-2\alpha)-2\alpha+s}\mathrm{e}^{-c|\xi|^{2-2\alpha}t}\right\|_{L^2}\|\widehat{\varphi}_2\|_{L^{\infty}}+\mathrm{e}^{-ct}\|(1-\chi_{\intt}(\xi))|\xi|^{s+1}\widehat{\varphi}_2\|_{L^2}\\
&\lesssim(1+t)^{-3-\frac{2s+n-4\alpha}{2(2-2\alpha)}}\|\varphi_2\|_{L^1}+\mathrm{e}^{-ct}\|\varphi_2\|_{\dot{H}^{s+1}}
\end{align*}
and
\begin{align*}
\|\chi_{\intt}(\xi)|\xi|^s\partial_t^3\widehat{\varphi}(t,\xi)\|_{L^2}&\lesssim \sup\limits_{|\xi|\leqslant\varepsilon_0}\left(|\xi|^{3(2-2\alpha)-2\alpha-1}\mathrm{e}^{-c|\xi|^{2-2\alpha}t}\right)\|\,|\xi|^{s+1}\widehat{\varphi}_2\|_{L^2}\\
&\lesssim(1+t)^{-3+\frac{2\alpha+1}{2-2\alpha}}\|\varphi_2\|_{\dot{H}^{s+1}}
\end{align*}
for any $s\geqslant0$. We completed the proof.
\end{proof}

\noindent \underline{\bf{2nd case: $\alpha\in(1/2,1]$}}

\begin{prop}\label{PROP-02}Let $\alpha\in(1/2,1]$ and $n\geqslant 3$. Let us consider the linearized viscous problem \eqref{Eq-Linear-Cattaneo-Fractional} with $(\varphi_0,\varphi_1,\varphi_2)\in \ml{A}_{s,j}^{(2)}$. Then, $\partial_t^j \varphi$ fulfills the upper bound estimates
	\begin{align}
		\|\partial_t^j\varphi(t,\cdot)\|_{\dot{H}^{s+2-j}}&\lesssim (1+t)^{-\frac{2s+n+2}{4\alpha}}\|(\varphi_0,\varphi_1,\varphi_2)\|_{ \ml{A}_{s,j}^{(2)}},\label{Est-05}\\
		\left\|\left(\partial_t^j\varphi-\ml{G}_{2,j}P_{\varphi_1+\tau\varphi_2}\right)(t,\cdot)\right\|_{\dot{H}^{s+2-j}}&=o\left(t^{-\frac{2s+n+2}{4\alpha}}\right),\label{Est-06}
	\end{align}
	 where the profiles $\ml{G}_{2,j}$ and the spaces $\ml{A}_{s,j}^{(2)}$ were introduced in \eqref{mlG2} and \eqref{A2-SPACE}, respectively. Moreover,  the upper bound estimates in \eqref{Est-05} are optimal since the following lower bound estimates hold:
	\begin{align*}
		\|\partial_t^j\varphi(t,\cdot)\|_{\dot{H}^{s+2-j}}\gtrsim t^{-\frac{2s+n+2}{4\alpha}}|P_{\varphi_1+\tau\varphi_2}|
	\end{align*}
	for $t\gg1$, provided that $|P_{\varphi_1+\tau\varphi_2}|\neq0$. Here, we take either $s=j-2$ or $s\geqslant0$ with $j=0,1,2$. On the other hand, assuming $(\varphi_0,\varphi_1,\varphi_2)\in\ml{A}_{s+1-2\alpha,2}^{(2)}$, the second-order derivative of the solution fulfills
	\begin{align}\label{C6}
	\|\partial_t^2\varphi(t,\cdot)\|_{\dot{H}^{s+1-2\alpha}}&\lesssim(1+t)^{-\frac{2s+n+4-4\alpha}{4\alpha}}\|(\varphi_0,\varphi_1,\varphi_2)\|_{\ml{A}_{s+1-2\alpha,2}^{(1)}},\\
	\left\|\left(\partial_t^2\varphi-\ml{G}_{2,2}P_{\varphi_1+\tau\varphi_2}\right)(t,\cdot)\right\|_{\dot{H}^{s+1-2\alpha}}&=o\left(t^{-\frac{2s+n+4-4\alpha}{4\alpha}}\right),\notag
	\end{align}
and
\begin{align*}
	\|\partial_t^2\varphi(t,\cdot)\|_{\dot{H}^{s+1-2\alpha}}\gtrsim t^{-\frac{2s+n+4-4\alpha}{4\alpha}}|P_{\varphi_1+\tau\varphi_2}|,
\end{align*}
for $t\gg1$, provided that $|P_{\varphi_1+\tau\varphi_2}|\neq0$.
\end{prop}

\begin{proof}
Let $s=j-2$ or $s\geqslant0$ with $j=0,1,2$. By using the pointwise estimates in Cases 1 and 2 for $\alpha\in(1/2,1]$ in Section \ref{Subsection-Refine-Fourier}, with the same approach as the one in Proposition \ref{PROP-01}, we find
\begin{align*}
\|\chi_{\intt}(\xi)|\xi|^{s+2-j}\partial_t^j\widehat{\varphi}(t,\xi)\|_{L^2}&\lesssim\left\|\chi_{\intt}(\xi)|\xi|^{s+1}\mathrm{e}^{-c|\xi|^{2\alpha}t}\right\|_{L^2}\|(\widehat{\varphi}_0,\widehat{\varphi}_1,\widehat{\varphi}_2)\|_{(L^{\infty})^3}\\
&\lesssim (1+t)^{-\frac{2s+n+2}{4\alpha}}\|(\varphi_0,\varphi_1,\varphi_2)\|_{(L^{1})^3},
\end{align*}
and exponential decay rates for bounded and large frequencies. We stress that the $L^2$-regularities for the Cauchy data from the estimates for large frequencies complete the proof of \eqref{Est-05}.

 From \eqref{Est-03}, we know
\begin{align*}
	 \left\|\partial_t^j\varphi(t,\cdot)-\ml{G}_{2,j}(t,|D|)(\varphi_1+\tau\varphi_2)(\cdot)\right\|_{\dot{H}^{s+2-j}} \color{black}\lesssim (1+t)^{-\frac{2s+n+2}{4\alpha}-\frac{2\alpha-1}{2\alpha}}\|(\varphi_0,\varphi_1,\varphi_2)\|_{ \ml{A}_{s,j}^{(2)}}.
\end{align*}
Moreover, the preliminary argument in \eqref{Est-04} shows the large-time property
\begin{align*}
\|\ml{G}_{2,j}(t,|D|)(\varphi_1+\tau\varphi_2)(\cdot)-\ml{G}_{2,j}(t,\cdot)P_{\varphi_1+\tau\varphi_2}\|_{\dot{H}^{s+2-j}}=o\left(t^{-\frac{2s+n+2}{4\alpha}}\right).
\end{align*}
Thus, we immediately obtain \eqref{Est-06} by the triangle inequality. Finally, with the aid of Lemma \ref{Lemma-01} we find
\begin{align*}
\|\,|D|^{s+2-j}\ml{G}_{2,j}(t,\cdot)\|_{L^2}=\left\|\chi_{\intt}(\xi)|\xi|^{s+2-j}\frac{\partial_t^j\sin(c_0|\xi|t)}{c_0|\xi|}\mathrm{e}^{-\frac{b\nu}{2}|\xi|^{2\alpha}t}\right\|_{L^2}\gtrsim t^{-\frac{2s+n+2}{4\alpha}}
\end{align*}
for $t\gg1$, following the approach in Proposition \ref{PROP-01}, we complete the proof of the desired lower bound estimates. Eventually, the derivations of estimates for $\partial_t^2\varphi(t,\cdot)$ in the $\dot{H}^{s+1-2\alpha}$ are analogous to the previous cases.
\end{proof}
\begin{coro}\label{CORO-02}Let $\alpha\in(1/2,1]$ and $n\geqslant 3$. Let us consider the linearized viscous problem \eqref{Eq-Linear-Cattaneo-Fractional} with $\varphi_0=\varphi_1=0$ and $\varphi_2\in\dot{H}^s$. Then, the functions $\partial_t^j\varphi$ fulfill the upper bound  estimates
	\begin{align*}
	\|\partial_t^j\varphi(t,\cdot)\|_{\dot{H}^{s+2-j}}&\lesssim (1+t)^{-\frac{1}{2\alpha}}\|\varphi_2\|_{\dot{H}^s},
	\end{align*}
	 where we take either $s=j-2$ or $s\geqslant0$ with $j=0,1,2$. Additionally, assuming $\varphi_2\in\dot{H}^{s+1}\cap L^1$, the third-order derivative of the solution fulfills
	 \begin{align}
	 \|\partial_t^3\varphi(t,\cdot)\|_{\dot{H}^{s+1-2\alpha}}&\lesssim(1+t)^{-\frac{2s+n+6-4\alpha}{4\alpha}}\|\varphi_2\|_{\dot{H}^{s+1}\cap L^1}, \label{C10}\\
	 \|\partial_t^3\varphi(t,\cdot)\|_{\dot{H}^{s+1-2\alpha}}&\lesssim(1+t)^{1-\frac{1}{\alpha}}\|\varphi_2\|_{\dot{H}^{s+1}}, \label{C11}
	 \end{align}
for $s\geqslant0$.
\end{coro}

\noindent \underline{\bf{3rd case: $\alpha=1/2$}}

The following results can be proved with the same approach seen for the previous cases, therefore, we omit the details. 
\begin{prop}\label{PROP-03}Let $\alpha=1/2$ and $n\geqslant 2$. Let us consider the linearized viscous problem \eqref{Eq-Linear-Cattaneo-Fractional} with $(\varphi_0,\varphi_1,\varphi_2)\in \ml{A}_{s,j}^{(1)}$. Then, $\partial_t^j \varphi$ fulfills the upper bound  estimates
	\begin{align}
		\|\partial_t^j\varphi(t,\cdot)\|_{\dot{H}^{s+2-j}}&\lesssim (1+t)^{-(s+1)-\frac{n}{2}}\|(\varphi_0,\varphi_1,\varphi_2)\|_{ \ml{A}_{s,j}^{(1)}}, \label{Est-09}\\
		\left\|\left(\partial_t^j\varphi-\ml{G}_{3,j}P_{\varphi_1+\tau\varphi_2}\right)(t,\cdot)\right\|_{\dot{H}^{s+2-j}}&=o\left(t^{-(s+1)-\frac{n}{2}}\right),\notag
	\end{align}
	where the profiles $\ml{G}_{3,j}$ and the spaces $\ml{A}_{s,j}^{(1)}$ were introduced in \eqref{mlG3} and \eqref{A1-SPACE}, respectively. Moreover, the upper bound estimates in \eqref{Est-09} are optimal since the following lower bound estimates hold:
	\begin{align*}
		\|\partial_t^j\varphi(t,\cdot)\|_{\dot{H}^{s+2-j}}\gtrsim t^{-(s+1)-\frac{n}{2}}|P_{\varphi_1+\tau\varphi_2}|
	\end{align*}
	for $t\gg1$, provided that $|P_{\varphi_1+\tau\varphi_2}|\neq0$. Here, we take either $s=j-2$ or $s\geqslant0$ with $j=0,1,2$.
\end{prop}
\begin{coro}\label{CORO-03}Let $\alpha=1/2$ and $n\geqslant 2$. Let us consider the linearized viscous problem \eqref{Eq-Linear-Cattaneo-Fractional}  with $\varphi_0=\varphi_1=0$ and $\varphi_2\in\dot{H}^s$. Then, the functions $\partial_t^j\varphi$ fulfill the upper bound  estimates
	\begin{align*}
	\|\partial_t^j\varphi(t,\cdot)\|_{\dot{H}^{s+2-j}}&\lesssim(1+t)^{-1}\|\varphi_2\|_{\dot{H}^s},
	\end{align*}
where we take either $s=j-2$ or $s\geqslant0$ with $j=0,1,2$. Additionally, assuming $\varphi_2\in \dot{H}^{s+1}\cap L^1$, the third-order derivative of the solution fulfills
\begin{align*}
	\|\partial_t^3\varphi(t,\cdot)\|_{\dot{H}^{s+1}}&\lesssim (1+t)^{-(s+3)-\frac{n}{2}}\|\varphi_2\|_{\dot{H}^{s+1}\cap L^1},\\
	\|\partial_t^3\varphi(t,\cdot)\|_{\dot{H}^{s+1}}&\lesssim (1+t)^{-2}\|\varphi_2\|_{\dot{H}^{s+1}},
\end{align*}
for  $s\geqslant0$.
\end{coro}
\begin{remark}\label{Rem03}
Because of $\|\chi_{\extt}(D)f\|_{H^s}\approx\|\chi_{\extt}(D)f\|_{\dot{H}^s}$ the $H^s$ regularities for the initial data in Propositions \ref{PROP-01}, \ref{PROP-02} and \ref{PROP-03} can be replaced by the $\dot{H}^s$ regularities for any $s\in\mb{R}$. We will use this kind of estimates (with the Riesz potential spaces replacing the Bessel potential spaces in $\mathcal{A}^{(\ell)}_{s,j}$) in the study of the nonlinear problem \eqref{Eq-Cattaneo-Fractional} in the next section. 
\end{remark}
\begin{remark}
Let us summarize all derived estimates in this section. By assuming suitable $H^{\tilde{s}}\cap L^1$ regularities for the initial data, where  $\tilde{s}\geqslant 0$ depends both on the order of the Cauchy data and on the spatial regularity for $\partial_t^j \varphi$, and $|P_{\varphi_1+\tau\varphi_2}|\neq 0$, the following optimal decay estimates for the solution (and its time derivatives) to the linearized viscous Cattaneo-type model \eqref{Eq-Linear-Cattaneo-Fractional} hold:
\begin{align*}
\|\partial_t^j\varphi(t,\cdot)\|_{\dot{H}^{s+2-j}}\simeq\begin{cases}
t^{-(j+1)-\frac{2s+n-2j}{2(2-2\alpha)}}&\mbox{when}\ \ \alpha\in[0,1/2),\\
t^{-\frac{2s+n+2}{4\alpha}}&\mbox{when}\ \ \alpha\in[1/2,1],
\end{cases}
\end{align*}
for $s=j-2$ or $s\geqslant0$ with $j=0,1,2$, as well as
\begin{align*}
\|\partial_t^2\varphi(t,\cdot)\|_{\dot{H}^{s+1-2\alpha}}\simeq t^{-\frac{2s+n+4-4\alpha}{4\alpha}}\ \ \mbox{when}\ \ \alpha\in(1/2,1],
\end{align*}
for $s\geqslant0$. 
\end{remark}
\section{Global (in time) solutions to the nonlinear viscous Cattaneo-type model}\label{Section_GESDS}
In this section, we study the nonlinear viscous Cattaneo-type model \eqref{Eq-Cattaneo-Fractional} involving fractional Laplacians with $\alpha\in[0,1/2)$, $\alpha\in(1/2,1]$ and $\alpha=1/2$, respectively. We will prove Theorems \ref{Thm-01}, \ref{Thm-02} and \ref{Thm-03}, after some preparatory explanations regarding our approach.

\subsection{Anomalous diffusion case: $\alpha\in[0,1/2)$} \label{Subsection anomalous diff case}


 This subsection is organized as follows:
\begin{itemize}
\item We firstly introduce a nonlinear integral operator, whose fixed point will be the solution to our problem, and the time-weighted evolution space where this operator is defined. Afterwards, we explain the philosophy of our proof illustrating how to handle the nonlinear terms and what are the target inequalities that we have to prove in order to apply Banach fixed point theorem.
\item Then, we shall see how to employ tools from the harmonic analysis in order to handle the nonlinear terms and the Cauchy data (see Appendix \ref{Appendix} for the inequalities from the harmonic analysis that we are going to apply).
\item Subsequently, we prove the validity of the contraction principle (uniformly in time) for our nonlinear integral operator.
\item Finally, we show the optimality of the long-time estimates and we derive the asymptotic profiles of the solution and its time-derivatives.
\end{itemize}

\subsubsection{Philosophy of the proof}
For $n\geqslant 2$ and $\alpha\in[0,1/2)$ let us consider the time-weighted Sobolev space 
 \begin{align*}
 	X_s(T):=\ml{C}([0,T],H^{s+2})\cap \ml{C}^1([0,T],H^{s+1})\cap \ml{C}^2([0,T],H^s)
 \end{align*}for $T>0$ with the corresponding norm
\begin{align*}
\|\psi\|_{X_s(T)}:=\sup\limits_{t\in[0,T]}\left(\,\sum\limits_{j=0,1,2}(1+t)^{j+\frac{n-4\alpha}{2(2-2\alpha)}}\|\partial_t^j\psi(t,\cdot)\|_{L^2}+\sum\limits_{j=0,1,2}(1+t)^{j+1+\frac{2s+n-2j}{2(2-2\alpha)}}\|\partial_t^j\psi(t,\cdot)\|_{\dot{H}^{s+2-j}} \right),
\end{align*}
where $s>n/2-1$. Note that the weights are the reciprocal functions of the decay rates of the $(L^2\cap L^1)-L^2$ estimates and the decay rates of the $(\dot{H}^s\cap L^1)-\dot{H}^s$ estimates in Proposition \ref{PROP-01}.


Then, motivated by Duhamel's principle, we introduce the integral operator 
\begin{align}\label{A5}
N:\psi\in X_s(T)\to N\psi:=\psi^{\lin}+\psi^{\non},
\end{align}
in which $\psi^{\lin}=\psi^{\lin}(t,x)$ solves the linearized model \eqref{Eq-Linear-Cattaneo-Fractional}, and the function $\psi^{\non}=\psi^{\non}(t,x)$ is given by
\begin{align*}
\psi^{\non}(t,x):=\int_0^tK_2(t-\eta,x)\ast\partial_t\ml{N}_{\psi}(\eta,x)\mathrm{d}\eta.
\end{align*}
We recall that $K_2(t,x)$ is the kernel for the third initial data in the linearized viscous problem \eqref{Eq-Linear-Cattaneo-Fractional}, whose representation in the Fourier space was determined in \eqref{Solution-Formula}. In other words, $K_2=K_2(t,x)$ is the fundamental solution to the viscous problem \eqref{Eq-Linear-Cattaneo-Fractional} with initial data $\varphi_0=\varphi_1=0$ and $\varphi_2=\delta_0$ where $\delta_0$ stands for the Dirac distribution at $x=0$ with respect to spatial variables. Importantly, due to the facts that $\widehat{K}_2(0,|\xi|)=\partial_t\widehat{K}_2(0,|\xi|)=0$ and $\partial_t^2\widehat{K}_2(0,|\xi|)=1$, by using an integration by parts with respect to the time variable, we arrive at  the crucial representations for $\partial_t^j\psi^{\non}$ via
\begin{align}
\partial_t^j\psi^{\non}(t,x)&=-\partial_t^j K_2(t,x)\ast\left(\frac{B}{2Ac_0^2}|\psi_1(x)|^2+|\nabla\psi_0(x)|^2\right)+\int_0^t\partial_t^{j+1}K_2(t-\eta,x)\ast\ml{N}_{\psi}(\eta,x)\mathrm{d}\eta\label{Rep-non}
\end{align}
with $j=0,1$.  Moreover, the second-order derivative fulfills
\begin{align}
	\partial_t^2\psi^{\non}(t,x)&= \ml{N}_{\psi}(t,x) - \partial_t^2 K_2(t,x)\ast\left(\frac{B}{2Ac_0^2}|\psi_1(x)|^2+|\nabla\psi_0(x)|^2\right)\notag\\
	&\quad+\int_0^t\partial_t^{3}K_2(t-\eta,x)\ast\ml{N}_{\psi}(\eta,x)\mathrm{d}\eta,\label{Cn} 
\end{align}
in which an additional term appears.

Next, we demonstrate the global (in time) existence and uniqueness of the small data Sobolev solutions to the nonlinear viscous Cattaneo-type model \eqref{Eq-Cattaneo-Fractional} by proving the existence of a unique fixed point of $N$ in the space $X_s(T)$.

Our target is to prove that the following inequalities hold uniformly with respect to $T>0$ (meaning that the hidden multiplicative constants are independent of $T$):
\begin{align}
\|N\psi\|_{X_s(T)}&\lesssim \|(\psi_0,\psi_1,\psi_2)\|_{\ml{A}_{s,0}^{(1)}}+\|(\psi_0,\psi_1)\|_{H^{s+2}\times H^{s+1}}^2+\|\psi\|_{X_s(T)}^2,\label{Crucial-01}\\
\|N\psi-N\bar{\psi}\|_{X_s(T)}&\lesssim \|\psi-\bar{\psi}\|_{X_s(T)}\left(\|\psi\|_{X_s(T)}+\|\bar{\psi}\|_{X_s(T)}\right),\label{Crucial-02}
\end{align}
for any $\psi,\bar{\psi}\in X_s(T)$.
Let $\|(\psi_0,\psi_1,\psi_2)\|_{\ml{A}_{s,0}^{(1)}}=\epsilon>0$ be a small constant. Since  $H^{s+2}\times H^{s+1}$ contains the first two factors in $\ml{A}_{s,0}^{(1)}$, then \eqref{Crucial-01} and \eqref{Crucial-02} provide the existence of a unique Sobolev solution $\psi=N\psi\in X_s(T)$ thanks to the Banach fixed point theorem. Furthermore, $\psi$ can be globally prolonged in time since \eqref{Crucial-01} and \eqref{Crucial-02} hold uniformly with respect to $T$. The previous strategy for proving the global (in time) existence of solutions has been used for several evolution equations with power nonlinearities (see, for example, \cite{Ebert-Reissig-book}).

\subsubsection{Preliminary and a priori estimates} \label{Subsection: preliminary and apriori est }

\underline{\bf{Estimates for the initial data:}}

From the first term in \eqref{Rep-non} we see that nonlinear terms depending on the Cauchy data appear. For this reason, we estimate these terms in the spaces $L^2\cap L^1$ and $\dot{H}^s\cap L^1$.

 Applying the fractional Gagliardo-Nirenberg inequality in Lemma \ref{fractionalgagliardonirenbergineq}, we may obtain
\begin{align*}
\|\,|\nabla\psi_0|^2\|_{L^2\cap L^1}+\|\,|\psi_1|^2\|_{L^2\cap L^1}&\lesssim \|\psi_0\|_{\dot{H}^1}^2+\|\psi_0\|_{\dot{H}^1_4}^2+\|\psi_1\|_{L^2}^2+\|\psi_1\|_{L^4}^2\\
&\lesssim \|\psi_0\|_{\dot{H}^1}^2+\|\psi_0\|_{L^2}^{2-\frac{n+4}{2(s+2)}}\|\psi_0\|_{\dot{H}^{s+2}}^{\frac{n+4}{2(s+2)}}+\|\psi_1\|_{L^2}^2+\|\psi_1\|_{L^2}^{2-\frac{n}{2(s+1)}}\|\psi_1\|_{\dot{H}^{s+1}}^{\frac{n}{2(s+1)}}\\
&\lesssim\|\psi_0\|_{H^{s+2}}^2+\|\psi_1\|_{H^{s+1}}^2
\end{align*}
and
\begin{align*}
\|\,|\nabla\psi_0|^2\|_{\dot{H}^s\cap L^1}+\|\,|\psi_1|^2\|_{\dot{H}^s\cap L^1}&\lesssim \|\nabla\psi_0\|_{L^{\infty}}\|\psi_0\|_{\dot{H}^{s+1}}+\|\psi_1\|_{L^{\infty}}\|\psi_1\|_{\dot{H}^s}+\|\psi_0\|_{\dot{H}^1}^2+\|\psi_1\|_{L^2}^2\\
&\lesssim\|\psi_0\|_{H^{s+2}}^2+\|\psi_1\|_{H^{s+1}}^2
\end{align*}
by the Sobolev embedding $H^{s+1}\hookrightarrow L^{\infty}$ as $s>n/2-1$ for all $n\geqslant 2$ and Lemmas  \ref{fractionembedd}-\ref{Sickel lemma}.

\medskip
\noindent\underline{\bf{Estimates for the nonlinear terms:}}

 Due to the integral term in \eqref{Rep-non} it is convenient to estimate $|\psi_t(\eta,\cdot)|^2$ and $|\nabla\psi(\eta,\cdot)|^2$ in the $L^2$, $\dot{H}^{s+1}$ norms, respectively. Lemma \ref{fractionalgagliardonirenbergineq} implies
\begin{align*}
\|\,|\psi_t(\eta,\cdot)|^2\|_{L^2}\lesssim\|\psi_t(\eta,\cdot)\|_{L^4}^2&\lesssim\|\psi_t(\eta,\cdot)\|_{L^2}^{2-\frac{n}{2(s+1)}}\|\psi_t(\eta,\cdot)\|_{\dot{H}^{s+1}}^{\frac{n}{2(s+1)}}\\
&\lesssim (1+\eta)^{-2-\frac{3n-8\alpha}{2(2-2\alpha)}}\|\psi\|_{X_s(\eta)}^2,
\end{align*}
where we used the condition $\frac{n}{4(s+1)}\in[0,1]$ which is always  true thanks to the assumption $s>n/2-1$. We make use of Lemmas \ref{fractionleibnizrule} as well as \ref{fractionembedd} to deduce
\begin{align*}
	\|\,|\psi_t(\eta,\cdot)|^2\|_{\dot{H}^{s+1}}&\lesssim\|\psi_t(\eta,\cdot)\|_{L^{\infty}}\|\psi_t(\eta,\cdot)\|_{\dot{H}^{s+1}}\\
	&\lesssim\left(\|\psi_t(\eta,\cdot)\|_{\dot{H}^{\frac{n}{2}-\varepsilon_1}}+\|\psi_t(\eta,\cdot)\|_{\dot{H}^{s+1}}\right)\|\psi_t(\eta,\cdot)\|_{\dot{H}^{s+1}}\\
	&\lesssim\|\psi_t(\eta,\cdot)\|_{L^2}^{1-\frac{n-2\varepsilon_1}{2(s+1)}}\|\psi_t(\eta,\cdot)\|_{\dot{H}^{s+1}}^{1+\frac{n-2\varepsilon_1}{2(s+1)}}+\|\psi_t(\eta,\cdot)\|_{\dot{H}^{s+1}}^2\\
	&\lesssim(1+\eta)^{-2-\frac{2s+3n+2-8\alpha-2\varepsilon_1}{2(2-2\alpha)}}\|\psi\|_{X_s(\eta)}^2
\end{align*}
with a sufficiently small constant $\varepsilon_1>0$. By using the same approach as before, we obtain
\begin{align}
\|\,|\nabla\psi(\eta,\cdot)|^2\|_{L^1}&\lesssim (1+\eta)^{-\frac{n-4\alpha+2}{2-2\alpha}}\|\psi\|_{X_s(\eta)}^2,\label{D3}\\
\|\,|\nabla\psi(\eta,\cdot)|^2\|_{L^2}&\lesssim (1+\eta)^{-\frac{3n-8\alpha+4}{2(2-2\alpha)}}\|\psi\|_{X_s(\eta)}^2,\label{D4}\\
\|\,|\nabla\psi(\eta,\cdot)|^2\|_{\dot{H}^{s+1}}&\lesssim (1+\eta)^{-2-\frac{2s+3n-2-2\varepsilon_1}{2(2-2\alpha)}}\|\psi\|_{X_s(\eta)}^2\label{D5}.
\end{align}
The last three estimates for $|\nabla\psi(\eta,\cdot)|^2$ play dominant roles comparing with the corresponding estimates for $|\psi_t(\eta,\cdot)|^2$ in $L^1$, $L^2$ and $\dot{H}^{s+1}$ norms.

\subsubsection{Uniform estimates and global (in time) existence of the solutions}
Let us begin with the estimates for $\partial_t^j N\psi$ when $j=0,1$. We recall the definition of $N\psi$ in \eqref{A5} together with the representation in \eqref{Rep-non} for $\partial_t^j \psi^{\mathrm{non}}$. by applying $(L^2\cap L^1)-L^2$ estimates (Proposition \ref{PROP-01} with $s=j-2$) in the sub-interval $[0,t/2]$ and $L^2-L^2$ estimates (Corollary \ref{CORO-01} with $s=j-2$) in the sub-interval $[t/2,t]$, we may derive
\begin{align*}
\|\partial_t^jN\psi(t,\cdot)\|_{L^2}&\lesssim(1+t)^{-j-\frac{n-4\alpha}{2(2-2\alpha)}}\left(\|(\psi_0,\psi_1,\psi_2)\|_{\ml{A}_{s,j}^{(1)}}+\left\||\psi_1|^2+|\nabla\psi_0|^2\right\|_{L^2\cap L^1}\right)\\
&\quad+\int_0^{t/2}(1+t-\eta)^{-(j+1)-\frac{n-4\alpha}{2(2-2\alpha)}}\left(\|\,|\psi_t(\eta,\cdot)|^2\|_{L^2\cap L^1}+\|\,|\nabla\psi(\eta,\cdot)|^2\|_{L^2\cap L^1}\right)\mathrm{d}\eta\\
&\quad+\int_{t/2}^t(1+t-\eta)^{-j}\left(\|\,|\psi_t(\eta,\cdot)|^2\|_{L^2}+\|\,|\nabla\psi(\eta,\cdot)|^2\|_{L^2}\right)\mathrm{d}\eta.
\end{align*} 
We remark that due to the higher order derivatives $\partial_t^{j+1}K_2$ with $j=0,1$, we have to use the estimates \eqref{A2} and \eqref{extra L2 L2 estimate alpha< 1/2} for $\partial_t K_2$ and $\partial_t^2K_2$ from Proposition \ref{PROP-01} and Corollary \ref{CORO-01}. In particular, we combined $L^2\hookrightarrow H^{j-1}$ with \eqref{A2} for $\eta\in [0,t/2]$.

From  \eqref{D3}, \eqref{D4},  we arrive at
\begin{align}
&(1+t)^{j+\frac{n-4\alpha}{2(2-2\alpha)}}\|\partial_t^jN\psi(t,\cdot)\|_{L^2}\notag\\
&\qquad\lesssim\|(\psi_0,\psi_1,\psi_2)\|_{\ml{A}_{s,j}^{(1)}}+\|\psi_0\|_{H^{s+2}}^2+\|\psi_1\|_{H^{s+1}}^2+(1+t)^{-1}\int_0^{t/2}(1+\eta)^{-\frac{n-4\alpha+2}{2-2\alpha}}\mathrm{d}\eta\,\|\psi\|_{X_s(T)}^2\notag\\
&\qquad\quad+(1+t)^{j-1-\frac{n}{2-2\alpha}}\int_{t/2}^t(1+t-\eta)^{-j}\mathrm{d}\eta\, \|\psi\|_{X_s(T)}^2\notag\\
&\qquad\lesssim\|(\psi_0,\psi_1,\psi_2)\|_{\ml{A}_{s,j}^{(1)}}+\|(\psi_0,\psi_1)\|_{H^{s+2}\times H^{s+1}}^2+\left((1+t)^{-1}+(1+t)^{-\frac{n}{2-2\alpha}+\varepsilon_1}\right)\|\psi\|_{X_s(T)}^2,\label{Est-07}
\end{align}
where we used $-\frac{n-4\alpha+2}{2-2\alpha}< -1$ since $n>2\alpha$ and $-(j+2)+\frac{j+1}{2-2\alpha}<-1$ since $\alpha\in[0,1/2)$. Hereafter, $\varepsilon_1>0$ is a suitably small constant. Furthermore, due to $j\leqslant\frac{n}{2-2\alpha}$ for $j=0,1$ and $n\geqslant 2$, we have
\begin{align*}
	(1+t)^{j+\frac{n-4\alpha}{2(2-2\alpha)}}\|\partial_t^jN\psi(t,\cdot)\|_{L^2}\lesssim\|(\psi_0,\psi_1,\psi_2)\|_{\ml{A}_{s,j}^{(1)}}+\|(\psi_0,\psi_1)\|_{H^{s+2}\times H^{s+1}}^2+\|\psi\|_{X_s(T)}^2.
\end{align*}

Analogously by \eqref{D3} and \eqref{D5}, the higher order energy terms can be estimated as follows:
\begin{align}
&(1+t)^{j+1+\frac{2s+n-2j}{2(2-2\alpha)}}\|\partial_t^jN\psi(t,\cdot)\|_{\dot{H}^{s+2-j}}\notag\\
&\qquad\lesssim\|(\psi_0,\psi_1,\psi_2)\|_{\ml{A}_{s,j}^{(1)}}+\|\psi_0\|_{H^{s+2}}^2+\|\psi_1\|_{H^{s+1}}^2+(1+t)^{-\frac{1-2\alpha}{2-2\alpha}}\int_0^{t/2}(1+\eta)^{-\frac{n-4\alpha+2}{2-2\alpha}}\mathrm{d}\eta\,\|\psi\|_{X_s(T)}^2\notag\\
&\qquad\quad+(1+t)^{j-1-\frac{n+j-1-\varepsilon_1}{2-2\alpha}}\int_{t/2}^t(1+t-\eta)^{-(j+2)+\frac{j+1}{2-2\alpha}}\mathrm{d}\eta\,\|\psi\|_{X_s(T)}^2\notag\\
&\qquad\lesssim\|(\psi_0,\psi_1,\psi_2)\|_{\ml{A}_{s,j}^{(1)}}+\|(\psi_0,\psi_1)\|_{H^{s+2}\times H^{s+1}}^2+(1+t)^{-\frac{1-2\alpha}{2-2\alpha}}\left(1+(1+t)^{j-\frac{n+j-\varepsilon_1}{2-2\alpha}}\right)\|\psi\|_{X_s(T)}^2\notag\\
&\qquad\lesssim\|(\psi_0,\psi_1,\psi_2)\|_{\ml{A}_{s,j}^{(1)}}+\|(\psi_0,\psi_1)\|_{H^{s+2}\times H^{s+1}}^2+\|\psi\|_{X_s(T)}^2.\label{Est-08}
\end{align}

 Let us turn to the second-order time-derivative. By the $(L^2\cap L^1)-L^2$ estimates in \eqref{C1} and the $L^2-L^2$ estimates in \eqref{C2}, we get
\begin{align*}
&(1+t)^{3+\frac{2s+n-4}{2(2-2\alpha)}}\|\partial_t^2N\psi(t,\cdot)\|_{\dot{H}^s}\\
&\qquad\lesssim\|(\psi_0,\psi_1,\psi_2)\|_{\ml{A}_{s,j}^{(1)}}+(1+t)^{3+\frac{2s+n-4}{2(2-2\alpha)}}\left\||\psi_t(t,\cdot)|^2+|\nabla\psi(t,\cdot)|^2\right\|_{\dot{H}^s}+\|(\psi_0,\psi_1)\|_{H^{s+2}\times H^{s+1}}^2\\
&\qquad\quad+(1+t)^{3+\frac{2s+n-4}{2(2-2\alpha)}}\int_0^{t/2}(1+t-\eta)^{-3-\frac{2s+n-4\alpha}{2(2-2\alpha)}}(1+\eta)^{-\frac{n-4\alpha+2}{2-2\alpha}}\mathrm{d}\eta\,\|\psi\|_{X_s(T)}^2\\
&\qquad\quad+(1+t)^{3+\frac{2s+n-4}{2(2-2\alpha)}}\int_{t/2}^t(1+t-\eta)^{-2-\frac{1-4\alpha}{2-2\alpha}}(1+\eta)^{-2-\frac{2s+3n-2-2\varepsilon_1}{2(2-2\alpha)}}\mathrm{d}\eta\,\|\psi\|_{X_s(T)}^2.
\end{align*} 

So, by using $(1+t-\eta)\approx (1+t)$ when $\eta\in[0,t/2]$ and $(1+\eta)\approx (1+t)$ when $\eta\in[t/2,t]$, we derive
\begin{align*}
	&(1+t)^{3+\frac{2s+n-4}{2(2-2\alpha)}}\|\partial_t^2N\psi(t,\cdot)\|_{\dot{H}^s}\\
	&\qquad\lesssim \|(\psi_0,\psi_1,\psi_2)\|_{\ml{A}_{s,2}^{(1)}}+\|(\psi_0,\psi_1)\|_{H^{s+2}\times H^{s+1}}^2+(1+t)^{-\frac{n-2+2\alpha-\varepsilon_1}{2-2\alpha}}\|\psi\|_{X_s(T)}^2\\
	&\qquad\quad+(1+t)^{-1}\int_0^{t/2}(1+\eta)^{-\frac{n-4\alpha+2}{2-2\alpha}}\mathrm{d}\eta\,\|\psi\|_{X_s(T)}^2\\
	&\qquad\quad+(1+t)^{-\frac{n+2\alpha-1-\varepsilon_1}{2-2\alpha}}\int_{t/2}^t(1+t-\eta)^{-2-\frac{1-4\alpha}{2-2\alpha}}\mathrm{d}\eta\,\|\psi\|_{X_s(T)}^2\\
	&\qquad\lesssim \|(\psi_0,\psi_1,\psi_2)\|_{\ml{A}_{s,2}^{(1)}}+\|(\psi_0,\psi_1)\|_{H^{s+2}\times H^{s+1}}^2+(1+t)^{-\min\left\{1,\frac{n+2\alpha-2-\varepsilon_1}{2-2\alpha}\right\}}\|\psi\|_{X_s(T)}^2\\
	&\qquad\lesssim \|(\psi_0,\psi_1,\psi_2)\|_{\ml{A}_{s,2}^{(1)}}+\|(\psi_0,\psi_1)\|_{H^{s+2}\times H^{s+1}}^2+\|\psi\|_{X_s(T)}^2.
\end{align*}
We point out explicitly that the previous estimates for $\|\partial_t^2N\psi(t,\cdot)\|_{\dot{H}^s}$ can be repeated in the same manner also for $\|\partial_t^2N\psi(t,\cdot)\|_{L^2}$ taking formally $s=0$.
According to $\ml{A}_{\sigma,k}^{(1)}\subset \ml{A}_{s,0}^{(1)}$ for $k=0,1,2$ and $\sigma=0,s$, combining all previous estimates, we complete the proof of the desired estimate in \eqref{Crucial-01}.

Now, we sketch the main steps in the derivation of \eqref{Crucial-02}, emphasizing the differences with the proof of \eqref{Crucial-01}. We begin by remarking that for $j=0,1$
\begin{align*}
\partial_t^j N\psi(t,x)- \partial_t^j N\bar{\psi}(t,x)=\int_0^t\partial_t^{j+1}K_2(t-\eta,x)\ast\left(\ml{N}_{\psi}(\eta,x)-\ml{N}_{\bar{\psi}}(\eta,x)\right)\mathrm{d}\eta
\end{align*}
while for the second-order derivative, we have an additional term besides the integral term, namely,
\begin{align*}
	\partial_t^2 N\psi(t,x)- \partial_t^2 N\bar{\psi}(t,x)=\ml{N}_{\psi}(t,x)-\ml{N}_{\bar{\psi}}(t,x)+ \int_0^t\partial_t^{j+1}K_2(t-\eta,x)\ast\left(\ml{N}_{\psi}(\eta,x)-\ml{N}_{\bar{\psi}}(\eta,x)\right)\mathrm{d}\eta
\end{align*}

 Observing that
\begin{align*}
\left||\psi_t|^2-|\bar{\psi}_t|^2\right|+\left||\nabla\psi|^2-|\nabla\bar{\psi}|^2\right|&\lesssim \left(|\psi_t|+|\bar{\psi}_t|\right)|\psi_t-\bar{\psi}_t|+ \left(|\nabla \psi|+|\nabla \bar{\psi}|\right)|\nabla \psi-\nabla \bar{\psi}|,
\end{align*}
employing the same approach used in Section \ref{Subsection: preliminary and apriori est } together with H\"older's inequality, we find
\begin{align*}
\left\|\ml{N}_{\psi}(\eta,\cdot)-\ml{N}_{\bar{\psi}}(\eta,\cdot)\right\|_{L^1}&\lesssim (1+\eta)^{-\frac{n-4\alpha+2}{2-2\alpha}}\left(\|\psi\|_{X_s(T)}+\|\bar{\psi}\|_{X_s(T)}\right)\|\psi-\bar{\psi}\|_{X_s(T)},\\
\left\|\ml{N}_{\psi}(\eta,\cdot)-\ml{N}_{\bar{\psi}}(\eta,\cdot)\right\|_{L^2}&\lesssim (1+\eta)^{-\frac{3n-8\alpha+4}{2(2-2\alpha)}}\left(\|\psi\|_{X_s(T)}+\|\bar{\psi}\|_{X_s(T)}\right)\|\psi-\bar{\psi}\|_{X_s(T)},\\
\left\|\ml{N}_{\psi}(\eta,\cdot)-\ml{N}_{\bar{\psi}}(\eta,\cdot)\right\|_{\dot{H}^{s+1}}&\lesssim
(1+\eta)^{-2-\frac{2s+3n-2-2\varepsilon_1}{2(2-2\alpha)}}
\left(\|\psi\|_{X_s(T)}+\|\bar{\psi}\|_{X_s(T)}\right)\|\psi-\bar{\psi}\|_{X_s(T)}.
\end{align*}
 Eventually, employing the same type of estimates used in the derivation of \eqref{Crucial-01} (namely, $(\dot{H}^{s}\cap L^1)-\dot{H}^s$ and $(L^2\cap L^1)-L^2$ estimates in $[0,t/2]$ and $\dot{H}^s-\dot{H}^s$ and $L^2-L^2$ estimates in $[t/2,t]$ from Proposition \ref{PROP-01} and Corollary \ref{CORO-01}, respectively) and using the uniform boundedness of the $\eta$-integrals seen in the previous computations, we conclude the proof of \eqref{Crucial-02}.

Hence, we proved the global (in time) existence of the small data Sobolev solutions in $ X_s(\infty)$. As a byproduct, from the smallness of the data, since the estimate
\begin{align*}
\|\psi\|_{X_s(T)}\lesssim\|(\psi_0,\psi_1,\psi_2)\|_{\ml{A}^{(1)}_{s,0}}+\|(\psi_0,\psi_1)\|_{H^{s+2}\times H^{s+1}}^2 \lesssim  \|(\psi_0,\psi_1,\psi_2)\|_{\ml{A}^{(1)}_{s,0}}
\end{align*}
holds for any $T>0$, we conclude the validity of the upper bound estimates in \eqref{Prof-02} for $\|\partial_t^j\psi(t,\cdot)\|_{\dot{H}^{s+2-j}}$, with $j=0,1,2$.

\subsubsection{Asymptotic profiles and lower bound estimates of the solutions}

In this subsection, given the global solution $\psi\in X_s(\infty)$ to \eqref{Eq-Cattaneo-Fractional}, we establish that $\mathcal{G}_{1,j}|\mb{B}_0|$ is the asymptotic profile of $\partial_t ^j \psi$ in $L^2$ and in $\dot{H}^{s+2-j}$ as $t\to \infty$. We recall that $\mb{B}_0$ consists of two terms.
From \eqref{def B0 profile}, \eqref{A5} and \eqref{Rep-non}, we obtain
\begin{align*}
\partial_t^j\psi-\ml{G}_{1,j}\mb{B}_0&=\left(\partial_t^j\psi^{\lin}-\ml{G}_{1,j}P_{\psi_1+\tau\psi_2}\right)+\int_0^t\partial_t^{j+1}K_2(t-\eta,x)\ast\ml{N}_{\psi}(\eta,x)\mathrm{d}\eta\\
&\quad-\left[\partial_t^jK_2(t,|D|)\left(\frac{B}{2Ac_0^2}|\psi_1|^2+|\nabla\psi_0|^2\right)-\tau\ml{G}_{1,j}P_{\frac{B}{2Ac_0^2}|\psi_1|^2+|\nabla\psi_0|^2}\right].
\end{align*}

From the error estimates for the linear part in \eqref{Prof-01}, we know that
\begin{align}\label{C5}
\left\|\left(\partial_t^j\psi^{\lin}-\ml{G}_{1,j}P_{\psi_1+\tau\psi_2}\right)(t,\cdot)\right\|_{\dot{H}^{s+2-j}}&=o\left(t^{-(j+1)-\frac{2s+n-2j}{2(2-2\alpha)}}\right)
\end{align}
as $t\gg1$. Taking vanishing first and second data in \eqref{err-01} and \eqref{err-02}, one arrives at
\begin{align}\label{B2}
&\left\|\left(\partial_t^jK_2(t,|D|)-\tau\ml{G}_{1,j}(t,|D|)\right)\left(\frac{B}{2Ac_0^2}|\psi_1|^2+|\nabla\psi_0|^2\right)\right\|_{\dot{H}^{s+2-j}}\notag\\
&\quad+\left\|\tau\ml{G}_{1,j}(t,|D|)\left(\frac{B}{2Ac_0^2}|\psi_1|^2+|\nabla\psi_0|^2\right)-\tau\ml{G}_{1,j}(t,\cdot)P_{\frac{B}{2Ac_0^2}|\psi_1|^2+|\nabla\psi_0|^2}\right\|_{\dot{H}^{s+2-j}}\notag\\
&=o\left(t^{-(j+1)-\frac{2s+n-2j}{2(2-2\alpha)}}\right)
\end{align}
for large-time $t\gg1$, where we used the assumption $(\psi_0,\psi_1)\in H^{s+2}\times H^{s+1}$. Moreover, in the chain of inequalities that led to  \eqref{Est-07} and \eqref{Est-08}, we already obtained
\begin{align}\label{B3}
&\left\|\int_0^t\partial_t^{j+1}K_2(t-\eta,\cdot)\ast\ml{N}_{\psi}(t,\cdot)\mathrm{d}\eta\right\|_{\dot{H}^{\sigma+2-j}}\notag\\
&\qquad\lesssim\begin{cases}
(1+t)^{-1}(1+t)^{-j-\frac{n-4\alpha}{2(2-2\alpha)}}\|(\psi_0,\psi_1,\psi_2)\|_{\ml{A}^{(1)}_{s,0}}^2&\mbox{when}\ \ \sigma=j-2,\\
(1+t)^{-\frac{1-2\alpha}{2-2\alpha}}(1+t)^{-(j+1)-\frac{2s+n-2j}{2(2-2\alpha)}}\|(\psi_0,\psi_1,\psi_2)\|_{\ml{A}^{(1)}_{s,0}}^2&\mbox{when}\ \ \sigma=s,
\end{cases}
\end{align}
where we used  $\|\psi\|_{X_s(T)}\lesssim \|(\psi_0,\psi_1,\psi_2)\|_{\ml{A}^{(1)}_{s,0}}$. Thus, we proved \eqref{Prof-01} for $j=0,1$. Applying  \eqref{Prof-01} together with Minkowski inequality and Lemma \ref{Lemma-01}, we find
\begin{align}
\|\partial_t^j\psi(t,\cdot)\|_{\dot{H}^{s+2-j}}&\gtrsim\|\ml{G}_{1,j}(t,\cdot)\|_{\dot{H}^{s+2-j}}|\mb{B}_0|- \left\|\left(\partial_t^j\psi-\ml{G}_{1,j}\mb{B}_0\right)(t,\cdot)\right\|_{\dot{H}^{s+2-j}}\notag\\
&\gtrsim t^{-(j+1)-\frac{2s+n-2j}{2(2-2\alpha)}}|\mb{B}_0|\label{C4}
\end{align}
as $t\gg1$ and $|\mb{B}_0|\neq0$, which is exactly \eqref{Prof-03} for $j=0,1$.

 For the second-order derivative of the solution, we can employ the same method that we just used for the lower order derivatives. Indeed, \eqref{C5} and \eqref{B2} are satisfied for $j=2$ as well. Then, it remains to estimate the additional term $\ml{N}_{\psi}(t,x)$ in \eqref{Cn} and the integral term. Choosing $\varepsilon_1 > 0$ sufficiently small, we actually already showed that
 
 \begin{align}
 \left\|\ml{N}_{\psi}(t,\cdot)\right\|_{\dot{H}^\sigma} & \lesssim (1+t)^{-\frac{n-2+2\alpha-\varepsilon_1}{2-2\alpha}} (1+t)^{-3-\frac{2s+n-4}{2(2-2\alpha)}}  \|(\psi_0,\psi_1,\psi_2)\|_{\ml{A}^{(1)}_{s,0}}^2,\notag\\
 \left\|\int_0^t\partial_t^{3}K_2(t-\eta,\cdot)\ast\ml{N}_{\psi}(\eta,\cdot)\mathrm{d}\eta\right\|_{\dot{H}^{\sigma}}&\lesssim(1+t)^{-\min\left\{1,\frac{n+2\alpha-1+2\varepsilon_1}{2-2\alpha}\right\}}(1+t)^{-3-\frac{2s+n-4}{2(2-2\alpha)}}\|(\psi_0,\psi_1,\psi_2)\|_{\ml{A}^{(1)}_{s,0}}^2,\label{C3}
 \end{align}
for $\sigma\in\{0,s\}$. 
Hence, \eqref{Prof-01}  is proved also for $j=2$, while the lower bound estimates can be derived similarly as we did in \eqref{C4}. The proof of Theorem \ref{Thm-01} is completed.

\begin{remark} \label{Remark alpha=0 and n=2}
In Theorem \ref{Thm-01}, we excluded the limit case $\alpha=0$ when $n=2$. This is due to the fact that, as we have seen in the proof, the $\dot{H}^s$ norm of the nonlinear term $\mathcal{N}_\psi$ appearing in $\partial_t^2 \psi^{\mathrm{non}}$, see \eqref{Cn}, grows exactly as the one for the solution to the linearized model and, therefore, we are not able to describe the asymptotic profile of $\partial_t^2 \psi$. Nevertheless, for $\alpha=0$ when $n=2$, taking $s>0$ and repeating the same kind of proof seen in this section, it is possible to prove that there exists a constant $\epsilon>0$ such that for all $(\psi_0,\psi_1,\psi_2)\in\ml{A}_{s,0}^{(1)}$ with $\|(\psi_0,\psi_1,\psi_2)\|_{\ml{A}_{s,0}^{(1)}}\leqslant\epsilon$, there is a uniquely determined Sobolev solution
\begin{align*}
	\psi\in\ml{C}([0,\infty),H^{s+2})\cap\ml{C}^1([0,\infty),H^{s+1})
\end{align*} that satisfies \eqref{Prof-02}, \eqref{Prof-01} and \eqref{Prof-03}  for $j=0,1$ and either $s_0=j-2$ or $s_0=s$.
\end{remark}

\subsection{Diffusion wave case: $\alpha\in(1/2,1]$}\label{Sub-GESES-DIFFUSION-WAVE}


 The approach for the proof of Theorem \ref{Thm-02} is analogous to the one seen in Section \ref{Subsection anomalous diff case} for Theorem \ref{Thm-01}. We sketch the main steps in the proof  of Theorem \ref{Thm-02}, highlighting the differences with the previous case. When $\alpha\in(1/2,1]$ and $n\geqslant 3$, we introduce for $T>0$ the weighted Sobolev space
\begin{align*}
Y_s(T):=\ml{C}([0,T],H^{s+2})\cap\ml{C}^1([0,T],H^{s+1})\cap\ml{C}^2([0,T],H^{s+1-2\alpha})
\end{align*}
equipped with the norm
\begin{align*}
\|\psi\|_{Y_s(T)}&:=\sup\limits_{t\in[0,T]}\bigg(\sum\limits_{j=0,1,2}(1+t)^{\frac{n-2+2j}{4\alpha}}\|\partial_t^j\psi(t,\cdot)\|_{L^2}+(1+t)^{\frac{2s+n+2}{4\alpha}}\sum\limits_{j=0,1}\|\partial_t^j\psi(t,\cdot)\|_{\dot{H}^{s+2-j}}\\
&\qquad\qquad\quad+(1+t)^{\frac{2s+n+4-4\alpha}{4\alpha}}\|\partial_t^2\psi(t,\cdot)\|_{\dot{H}^{s+1-2\alpha}}\bigg),
\end{align*}
where the regularity parameter $s$ satisfies $s>n/2-1$.

 Then, with the help of some fractional tools from the harmonic analysis (see the lemmas in Appendix \ref{Appendix}), similarly as in Section \ref{Subsection: preliminary and apriori est }, we may deduce the following estimates for the nonlinear terms:
\begin{align}
\|\,|\psi_t(\eta,\cdot)|^2\|_{L^1}+\|\,|\nabla\psi(\eta,\cdot)|^2\|_{L^1}&\lesssim(1+\eta)^{-\frac{n}{2\alpha}}\|\psi\|_{Y_s(\eta)}^2,\label{C7}\\
\|\,|\psi_t(\eta,\cdot)|^2\|_{L^2}+\|\,|\nabla\psi(\eta,\cdot)|^2\|_{L^2}&\lesssim (1+\eta)^{-\frac{3n}{4\alpha}}\|\psi\|_{Y_s(\eta)}^2,\label{C8}\\
\|\,|\psi_t(\eta,\cdot)|^2\|_{\dot{H}^{s+1}}+\|\,|\nabla\psi(\eta,\cdot)|^2\|_{\dot{H}^{s+1}}&\lesssim(1+\eta)^{-\frac{2s+3n+2-2\varepsilon_1}{4\alpha}}\|\psi\|_{Y_s(\eta)}^2,\label{C9}
\end{align}
with a sufficiently small constant $\varepsilon_1>0$. 
Let us begin with the estimates of $\partial_t^j N\psi$ for $j=0,1$. Combining with \eqref{C7}-\eqref{C9}  with the $(L^2\cap L^1)-L^2$ estimates from Proposition \ref{PROP-02} and the $L^2-L^2$ estimates from Corollary \ref{CORO-02}, we  derive
\begin{align}
&(1+t)^{\frac{n-2+2j}{4\alpha}}\|\partial_t^j\psi^{\non}(t,\cdot)\|_{L^2}\notag\\
&\qquad\lesssim\|\psi_0\|_{H^{s+2}}^2+\|\psi_1\|_{H^{s+2\alpha}}^2+(1+t)^{\frac{n-2+2j}{4\alpha}}\int_0^{t/2}(1+t-\eta)^{-\frac{n+2j}{4\alpha}}(1+\eta)^{-\frac{n}{2\alpha}}\mathrm{d}\eta\,\|\psi\|_{Y_s(T)}^2\notag\\
&\qquad\quad+(1+t)^{\frac{n-2+2j}{4\alpha}}\int_{t/2}^t(1+t-\eta)^{-\frac{1}{2\alpha}}(1+\eta)^{-\frac{3n}{4\alpha}}\mathrm{d}\eta\,\|\psi\|_{Y_s(T)}^2\notag\\
&\qquad\lesssim\|(\psi_0,\psi_1)\|_{H^{s+2}\times H^{s+2\alpha}}^2+(1+t)^{-\frac{1}{2\alpha}} \|\psi\|_{Y_s(T)}^2,\label{A6} 
\end{align}
due to $n>2\alpha$.
Analogously, by using the $(\dot{H}^{s}\cap L^1)-\dot{H}^{s+2-j}$ estimates from Proposition \ref{PROP-02} and the $\dot{H}^{s}-\dot{H}^{s+2-j}$ estimates from Corollary \ref{CORO-02}, we have
\begin{align}\label{A7}
(1+t)^{\frac{2s+n+2}{4\alpha}}\|\partial_t^j\psi^{\non}(t,\cdot)\|_{\dot{H}^{s+2-j}}\lesssim\|(\psi_0,\psi_1)\|_{H^{s+2}\times H^{s+2\alpha}}^2+(1+t)^{-\frac{1}{2\alpha}} \|\psi\|_{Y_s(T)}^2. 
\end{align}
In an analogous way, we get
\begin{align*}
(1+t)^{\frac{n+2}{4\alpha}}\|\partial_t^2\psi^{\non}(t,\cdot)\|_{L^2}&\lesssim\|(\psi_0,\psi_1)\|_{H^{s+2}\times H^{s+2\alpha}}^2+(1+t)^{-\frac{n-1}{2\alpha}}\|\psi\|_{Y_s(T)}^2+(1+t)^{-\frac{1}{2\alpha}}\|\psi\|_{Y_s(T)}^2.
\end{align*} 

Let us estimate for $\partial_t^2\psi^{\non}(t,\cdot)$ in $\dot{H}^{s+1-2\alpha}$. We stress that the next estimate is the reason why we have  a loss of regularity in comparison with the linearized problem. Applying \eqref{C6} for the term $-\partial_t^2 K_2(t,x)\ast \mathcal{N}_\psi(0,x)$, \eqref{C10} for $\eta\in [0,t/2]$ and \eqref{C11} for $\eta\in[t/2,t]$ we arrive at
\begin{align*}
	&(1+t)^{\frac{2s+n+4-4\alpha}{4\alpha}}\|\partial_t^2\psi^{\non}(t,\cdot)\|_{\dot{H}^{s+1-2\alpha}}\\
	&\qquad\lesssim (1+t)^{\frac{2s+n+4-4\alpha}{4\alpha}}\left\||\psi_t(t,\cdot)|^2+|\nabla\psi(t,\cdot)|^2\right\|_{\dot{H}^{s+1-2\alpha}}+\left\||\psi_1|^2+|\nabla\psi_0|^2\right\|_{\dot{H}^{s+1-2\alpha}\cap L^1}\\
	&\qquad\quad+(1+t)^{\frac{2s+n+4-4\alpha}{4\alpha}}\int_0^{t/2}(1+t-\eta)^{-\frac{2s+n+6-4\alpha}{4\alpha}}\left\||\psi_t(\eta,\cdot)|^2+|\nabla\psi(\eta,\cdot)|^2\right\|_{\dot{H}^{s+1}\cap L^1}\mathrm{d}\eta\\
	&\qquad\quad+(1+t)^{\frac{2s+n+4-4\alpha}{4\alpha}}\int_{t/2}^t(1+t-\eta)^{1-\frac{1}{\alpha}}\left\||\psi_t(\eta,\cdot)|^2+|\nabla\psi(\eta,\cdot)|^2\right\|_{\dot{H}^{s+1}}\mathrm{d}\eta.
\end{align*}
Thus, by \eqref{C7}-\eqref{C9} we get
\begin{align*}
	(1+t)^{\frac{2s+n+4-4\alpha}{4\alpha}}\|\partial_t^2\psi^{\non}(t,\cdot)\|_{\dot{H}^{s+1-2\alpha}}&\lesssim (1+t)^{-\frac{n-1-\varepsilon_1}{2\alpha}}\|\psi\|_{Y_s(T)}^2+\|(\psi_0,\psi_1)\|_{H^{s+2}\times H^{s+2\alpha}}^2\\
	&\quad+(1+t)^{-\frac{1}{2\alpha}}\|\psi\|_{Y_s(T)}^2+(1+t)^{-\frac{n-2\alpha+1-\varepsilon_1}{2\alpha}}\|\psi\|_{Y_s(T)}^2\\
	&\lesssim\|(\psi_0,\psi_1)\|_{H^{s+2}\times H^{s+2\alpha}}^2+(1+t)^{-\frac{1}{2\alpha}}\|\psi\|_{Y_s(T)}^2
\end{align*}
by taking $\varepsilon_1\downarrow0$. Therefore, combining the previous estimates, we complete the proofs \eqref{Crucial-01} and \eqref{Crucial-02} in $Y_s(T)$. Consequently, there exists a unique global (in time) solution in $Y_s(\infty)$ when $\alpha\in(1/2,1]$.

By \eqref{Est-06} we have
	\begin{align*}
		\left\|\left(\partial_t^j\psi^{\lin}-\ml{G}_{2,j}P_{\psi_1+\tau\psi_2}\right)(t,\cdot)\right\|_{\dot{H}^{s+2-j}}=o\left(t^{-\frac{n-2+2j}{4\alpha}}\right),
	\end{align*}
 furthermore, from \eqref{A6} and \eqref{A7} for $j=0,1$, it follows that 
\begin{align*}
	\left\|\int_0^t\partial_t^{j+1}K_2(t-\eta,\cdot)\ast\ml{N}_{\psi}(\eta,\cdot)\mathrm{d}\eta\right\|_{\dot{H}^{s+2-j}}\lesssim (1+t)^{-\frac{1}{2\alpha}}(1+t)^{-\frac{n-2+2j}{4\alpha}} \|\psi\|_{Y_s(T)}^2.
\end{align*}
Hence, we derived the asymptotic profile of $\partial_t \psi$ in Theorem \ref{Thm-02}. The profile for the second-order derivative of the solution may be derived in the same way.

 Eventually, thanks to the lower bound estimates
\begin{align*}
	\|\ml{G}_{2,j}(t,\cdot)\|_{\dot{H}^{s+2-j}}&\gtrsim t^{-\frac{n-2+2j}{4\alpha}},\\
	\|\ml{G}_{2,2}(t,\cdot)\|_{\dot{H}^{s+1-2\alpha}}&\gtrsim t^{-3-\frac{2s+n-4}{2(2-2\alpha)}},
\end{align*}
 for large-time $t\gg1$, we completed the proof of Theorem \ref{Thm-02}.
\subsection{Threshold case: $\alpha=1/2$}


For $n\geqslant 2$ we consider the threshold case $\alpha=1/2$. We introduce the time-weighted Sobolev space 
\begin{align*}
	Z_s(T):=\ml{C}([0,T],H^{s+2})\cap \ml{C}^1([0,T],H^{s+1})\cap \ml{C}^2([0,T],H^s)
\end{align*}for $T>0$ with the norm
\begin{align*}
	\|\psi\|_{Z_s(T)}:=\sup\limits_{t\in[0,T]}\left(\,\sum\limits_{j=0,1,2}(1+t)^{j-1+\frac{n}{2}}\|\partial_t^j\psi(t,\cdot)\|_{L^2}+(1+t)^{s+1+\frac{n}{2}}\sum\limits_{j=0,1,2}\|\partial_t^j\psi(t,\cdot)\|_{\dot{H}^{s+2-j}}\right),
\end{align*}
where $s>n/2-1$.

 By following the same steps as in the case $\alpha\in(1/2,1]$, we notice that the nonlinear terms satisfy the estimates \eqref{C7}-\eqref{C9} as well. Moreover, since in $X_s(T)$ (space for the solutions in the case $\alpha\in[0,1/2)$) and in $Z_s(T)$ the same regularity for $\partial_t^j \psi$ appear for $j=0,1,2$, we can repeat the same computations changing only the decay rates. For brevity, we omit the computational details for this proof.

\section{Blow-up of solutions to the nonlinear inviscid Cattaneo-type model} \label{Section blowup result}

\subsection{Construction of iteration frame}
Let us recall that $\Phi$ defined in \eqref{BL-01} is positive and smooth function, which solves $\Delta\Phi=\Phi$ and asymptotically satisfies
\begin{align*}
\Phi(x)\sim|x|^{-\frac{n-1}{2}}\mathrm{e}^{|x|}\ \ \mbox{as}\ \ |x|\to\infty.
\end{align*}
We introduce now the test function $$\Upsilon=\Upsilon(t,x):=\mathrm{e}^{-c_0t}\Phi(x),$$ which is a solution to the adjoint equation for the linearized inviscid Cattaneo-type model
\begin{align}\label{P1}
	(-\tau\partial_t+\ml{I})(\partial_t^2-c_0^2\Delta)\Upsilon=0,
\end{align}
and fulfills
\begin{align}\label{P8}
\int_{B_{R+c_0t}}\Upsilon(t,x)\mathrm{d}x\leqslant C_1(R+c_0t)^{\frac{n-1}{2}}
\end{align}
for any $t\geqslant 0$ and $n\geqslant 1$ (cf. \cite[Equation (3.5)]{Lai-Takamura}), where $C_1=C_1(n,R,c_0)$ is a positive constant.

 Since the third-order partial differential operator in the inviscid Cattaneo-type model  \eqref{Eq-Cattaneo-inviscid} admits the factorization $(\tau\partial_t+\ml{I})(\partial_t^2-c_0^2\Delta)$, assuming compactly supported initial data, that is $\mathrm{supp}\, \psi_j\subset B_R$ for $j=0,1,2$ and for some positive constant $R$, then a local (in time) energy solution $\psi=\psi(t,x)$ has support contained in the forward cone $\{(t,x)\in[0,T)\times\mb{R}^n:|x|\leqslant R+c_0t \}$. This support condition for the solution allows us to consider as test function in Definition \ref{Defn-energy-solution}  even $\phi\in \mathcal{C}^\infty ([0,T)\times\mathbb{R}^n )$, i.e. without requiring any compactness for the support of the test function $\phi$. 

Therefore, we take $\Upsilon$ as test function in \eqref{P0}. Since \eqref{P1} holds, we have
\begin{align}\label{P2}
&\int_{\mb{R}^n}\big(\tau\psi_{tt}(t,x)+(\tau c_0+1)\psi_t(t,x)+c_0\psi(t,x)\big)\Upsilon(t,x)\mathrm{d}x\notag\\
&\qquad=c_0\int_0^t\int_{\mb{R}^n}\ml{N}_{\psi}(\eta,x)\Upsilon(\eta,x)\mathrm{d}x\mathrm{d}\eta+\int_{\mb{R}^n}\ml{N}_{\psi}(t,x)\Upsilon(t,x)\mathrm{d}x+\epsilon_0\ml{L}_0-\epsilon_0^2\ml{L}_1, 
\end{align}
where we denoted
\begin{align*}
	\ml{L}_0&:=\int_{\mb{R}^n}\big(\tau\psi_2(x)+(\tau c_0+1)\psi_1(x)+c_0\psi_0(x)\big)\Phi(x)\mathrm{d}x,\\
	\ml{L}_1&:=\int_{\mb{R}^n}\left(\frac{B}{2Ac_0^2}|\psi_1(x)|^2+|\nabla\psi_0(x)|^2\right)\Phi(x)\mathrm{d}x.
\end{align*}
Let us introduce the following time-dependent functional:
\begin{align*}
\ml{D}_0(t):=\int_{\mb{R}^n}\psi_t(t,x)\Upsilon(t,x)\mathrm{d}x.
\end{align*}
Using $\ml{D}_0$, we may rewrite \eqref{P2} as
\begin{align}\label{P3}
	&\tau\ml{D}_0'(t)+(2\tau c_0+1)\ml{D}_0(t)+c_0\int_{\mb{R}^n}\psi(t,x)\Upsilon(t,x)\mathrm{d}x\notag\\
	&\qquad=c_0\int_0^t\int_{\mb{R}^n}\ml{N}_{\psi}(\eta,x)\Upsilon(\eta,x)\mathrm{d}x\mathrm{d}\eta+\int_{\mb{R}^n}\ml{N}_{\psi}(t,x)\Upsilon(t,x)\mathrm{d}x+\epsilon_0\ml{L}_0-\epsilon_0^2\ml{L}_1.
\end{align}
Differentiating \eqref{P3} with respect to $t$, we arrive at
\begin{align}\label{P4}
&\tau\ml{D}''_0(t)+(2\tau c_0+1)\ml{D}'_0(t)+c_0 \ml{D}_0(t)-c_0^2\int_{\mb{R}^n}\psi(t,x)\Upsilon(t,x)\mathrm{d}x\notag\\
&\qquad=\left(c_0+\frac{\mathrm{d}}{\mathrm{d}t}\right)\int_{\mb{R}^n}\ml{N}_{\psi}(t,x)\Upsilon(t,x)\mathrm{d}x.
\end{align}
Multiplying \eqref{P3} by $c_0$ and adding the resulting equality to \eqref{P4}, we find
\begin{align*}
	&\tau\ml{D}''_0(t)+(3\tau c_0+1)\ml{D}'_0(t)+2c_0(\tau c_0+1)\ml{D}_0(t)\notag\\
	&\qquad=c_0^2\int_0^t\int_{\mb{R}^n}\ml{N}_{\psi}(\eta,x)\Upsilon(\eta,x)\mathrm{d}x\mathrm{d}\eta+\left(2c_0+\frac{\mathrm{d}}{\mathrm{d}t}\right)\int_{\mb{R}^n}\ml{N}_{\psi}(t,x)\Upsilon(t,x)\mathrm{d}x+\epsilon_0 c_0\ml{L}_0-\epsilon_0^2 c_0\ml{L}_1.
\end{align*}
In order to derive an iteration frame for the functional $\ml{D}_0$, let us introduce a first auxiliary functional
\begin{align*}
\ml{D}_1(t)&:=\ml{D}'_0(t)+2 c_0\ml{D}_0(t)-\frac{\tau c_0-1}{\tau^2}\int_0^t\int_{\mb{R}^n}\ml{N}_{\psi}(\eta,x)\Upsilon(\eta,x)\mathrm{d}x\mathrm{d}\eta-\frac{1}{\tau}\int_{\mb{R}^n}\ml{N}_{\psi}(t,x)\Upsilon(t,x)\mathrm{d}x -\epsilon_0 I_0,
\end{align*} where
\begin{align*}
I_0:= \frac{c_0}{\tau c_0+1}\int_{\mb{R}^n}\Big[ \tau \psi_2(x)+ (\tau c_0+1)\psi_1(x)+c_0 \psi_0(x) -\epsilon_0 \Big(\frac{B}{2Ac_0^2}|\psi_1(x)|^2+|\nabla\psi_0(x)|^2\Big)\Big]\Phi(x)\mathrm{d}x.
\end{align*} 
 This functional satisfies
\begin{align*}
\tau\ml{D}_1'(t)+(\tau c_0+1)\ml{D}_1(t)\geqslant \frac{1}{\tau^2}\int_0^t\int_{\mb{R}^n}\ml{N}_{\psi}(\eta,x)\Upsilon(\eta,x)\mathrm{d}x\mathrm{d}\eta\geqslant 0,
\end{align*} since both the nonlinearity and the function $\Upsilon$ are nonnegative.
From the previous inequality we get immediately $\ml{D}_1(t)\geqslant \mathrm{e}^{-(c_0+\frac{1}{\tau})t}\ml{D}_1(0)$ for any $t\in [0,T)$. Thanks to \eqref{thm bu data 3} the initial value $\ml{D}_1(0)$ is nonnegative and, consequently, $\ml{D}_1(t)\geqslant 0$ as well. From the definition of $\ml{D}_1(t)$ it follows that
\begin{align}\label{P7}
& \ml{D}'_0(t)+2c_0 \ml{D}_0(t)\geqslant\frac{\tau c_0-1}{\tau^2}\int_0^t\int_{\mb{R}^n}\ml{N}_{\psi}(\eta,x)\Upsilon(\eta,x)\mathrm{d}x\mathrm{d}\eta+\frac{1}{\tau}\int_{\mb{R}^n}\ml{N}_{\psi}(t,x)\Upsilon(t,x)\mathrm{d}x +\epsilon_0 I_0.
\end{align}
Now, we introduce a second auxiliary functional
\begin{align*}
\ml{D}_2(t):=\ml{D}_0(t)-\frac{\tau c_0-1}{2 c_0\tau^2}\int_0^t\int_{\mb{R}^n}\ml{N}_{\psi}(\eta,x)\Upsilon(\eta,x)\mathrm{d}x\mathrm{d}\eta,
\end{align*}
which fulfills
\begin{align*}
	& \ml{D}'_2(t)+2c_0\ml{D}_2(t) \geqslant \frac{\tau c_0+1}{2 c_0\tau^2}\int_{\mb{R}^n}\ml{N}_{\psi}(t,x)\Upsilon(t,x)\mathrm{d}x +\epsilon_0 I_0.
\end{align*} We point out that the assumption in \eqref{thm bu data 2} from Theorem \ref{Thm-Blow-up} guarantees that $I_0>0$. Therefore, from the previous estimate we get $ \ml{D}'_2(t)+2c_0\ml{D}_2(t) \geqslant \epsilon_0 I_0$ which implies in turn that $$\ml{D}_2(t)\geqslant \frac{I_0 \, \epsilon_0}{2c_0}  (1-\mathrm{e}^{-2c_0 t})+\mathrm{e}^{-2c_0 t}\ml{D}_2(0)$$ for any $t\in[0,T)$. 
Because of $\ml{D}_2(0)=\ml{D}_0(0)=\epsilon_0 \int_{\mathbb{R}^n} \psi_1(x) \Phi(x) \mathrm{d}x\geqslant 0$, here we used \eqref{thm bu data 1}, we have just proved that 
\begin{align}\label{D2 nonnegative}
\ml{D}_2(t)\geqslant 0
\end{align}
for any $t\in [0,T)$. 
Being $\ml{D}_2$ a nonnegative function, it follows that
\begin{align}\label{P9}
\ml{D}_0(t)\geqslant \frac{\tau c_0-1}{2 c_0\tau^2}\int_0^t\int_{\mb{R}^n}\ml{N}_{\psi}(\eta,x)\Upsilon(\eta,x)\mathrm{d}x\mathrm{d}\eta.
\end{align} Moreover, employing again \eqref{D2 nonnegative}, for any $t\in [1,T)$ and for a suitable constant $C_2>0$ depending on the Cauchy data, we obtain that
\begin{align} \label{P11bis}
\ml{D}_0(t)\geqslant \ml{D}_2(t)\geqslant C_2 \epsilon_0.
\end{align} 
Combining H\"older's inequality, \eqref{P8} and \eqref{P9}, we find the iteration frame
\begin{align}\label{Iteration-Frame}
\ml{D}_0(t)\geqslant C_3 \int_0^t(R+c_0\eta)^{-\frac{n-1}{2}}(\ml{D}_0(\eta))^2\mathrm{d}\eta
\end{align}
for a suitable constant $C_3>0$. 
In order to prove the blow-up in finite time of $\ml{D}_0$ we consider separately the cases $n=1,2$ from the case $n=3$. Indeed, in the latter case we have a factor $(R+c_0 t)^{-1}$ so we have to handle logarithmic terms as well.

\subsection{Iteration argument when $n=1,2$}
Our blow-up result will be established through an iteration argument, by deriving a sequence of lower bounds for $\ml{D}_0(t)$ as follows:
\begin{align}\label{Seq-Need}
	\ml{D}_0(t)\geqslant Q_j(R+c_0t)^{-\beta_j}(t-1)^{\gamma_j}
\end{align}
for any $t\geqslant 1$ and $j\in\mathbb{N}_0$, where $\{Q_j\}_{j\in\mb{N}_0}$, $\{\beta_j\}_{j\in\mb{N}_0}$ and $\{\gamma_j\}_{j\in\mb{N}_0}$ are sequences of non-negative real numbers that will be determined in the inductive step. According to \eqref{P11bis}, \eqref{Seq-Need} is valid for $j=0$ with
\begin{align*}
	Q_0:=C_2\epsilon_0,\ \ \beta_0=0,\ \ \gamma_0=0.
\end{align*}
Next, let us assume that \eqref{Seq-Need} holds for $j\geqslant0$ and prove it for $j+1$. Plugging \eqref{Seq-Need} into \eqref{Iteration-Frame}, we obtain
\begin{align*}
\ml{D}_0(t)&\geqslant C_3 Q_j^2\int_1^t(R+c_0\eta)^{-\frac{n-1}{2}-2\beta_j}(\eta-1)^{2\gamma_j}\mathrm{d}\eta\\
&\geqslant \frac{C_3Q_j^2}{2\gamma_j+1}(R+c_0t)^{-\frac{n-1}{2}-2\beta_j}(t-1)^{2\gamma_j+1}
\end{align*} 
for any $t\geqslant 1$. In other words, we obtained \eqref{Seq-Need} for $j+1$ assuming that
\begin{align*}
Q_{j+1}:=\frac{C_3}{2\gamma_j+1}Q_j^2,\ \ \beta_{j+1}:=2\beta_j+\frac{n-1}{2},\ \ \gamma_{j+1}:=2\gamma_j+1.
\end{align*}
By applying recursively the previous relations, we find
\begin{align}
\beta_{j}=2^{j}\beta_0+\frac{n-1}{2}(2^j-1)=\frac{n-1}{2}(2^j-1),\ \ 
\gamma_j=2^j\gamma_0+2^j-1=2^{j}-1.\label{P15}
\end{align}
The representations of $\gamma_j$ and $Q_j$ lead to $Q_j\geqslant C_32^{-j}Q_{j-1}^2$. Hence, applying the logarithmic function to both sides of the last inequality and using the resulting relation iteratively, we have
\begin{align*}
\ln Q_j&\geqslant \ln C_3-j\ln 2+2\ln Q_{j-1}\notag\\
&\geqslant\cdots\geqslant\sum\limits_{k=0}^{j-1}2^k\ln C_3-\sum\limits_{k=0}^{j-1}(j-k)2^k\ln 2+2^j\ln Q_0\notag\\
&\geqslant 2^j\left(\ln C_3-2\ln 2+\ln Q_0\right)+j\ln 2+2\ln 2-\ln C_3.
\end{align*}
Choosing $j_0=j_0(n)\in\mb{N}$ to be the smallest non-negative integer such that $j_0\geqslant\ln C_3 /\ln 2-2$, for any $j\geqslant j_0$ we obtain
\begin{align} \label{P16}
	\ln Q_j\geqslant 2^j\ln \left(\frac{C_3C_2}{4}\epsilon_0\right)=:2^j\ln (E_0\epsilon_0).
\end{align}
Combining \eqref{Seq-Need}, \eqref{P15} and \eqref{P16}, we get
\begin{align*}
	\ml{D}_0(t)&\geqslant \exp\big(2^j\ln (E_0\epsilon_0)\big)(R+c_0t)^{-\frac{n-1}{2}(2^j-1)}(t-1)^{2^{j}-1}\\
	&\geqslant \exp\left[2^j\left(\ln(E_0\epsilon_0)-\tfrac{n-1}{2}\ln(R+c_0t)+\ln(t-1)\right)\right](R+c_0t)^{-\frac{n-1}{2}}(t-1)^{-1}
\end{align*}
for any $j\geqslant j_0$ and any $t\geqslant 1$.
\newline
Taking $t\geqslant\max\{R/c_0,2\}$ so that $2c_0t\geqslant R+c_0t$ and $t-1\geqslant t/2$, one finds
\begin{align}\label{P10}
\ml{D}_0(t)\geqslant\exp\left[2^j\ln\left(E_1\epsilon_0\,t^{\frac{3-n}{2}}\right)\right](R+c_0t)^{-\frac{n-1}{2}}(t-1)^{-1}
\end{align}
with $E_1:=E_02^{-(n+1)}c_0^{-(n-1)}>0$. Thus, the power in $t^{\frac{3-n}{2}}$ is positive when $n=1,2$, and is zero when $n=3$ (for this reason we discuss the three-dimensional case in the next subsection). Let us fix $\epsilon_1=\epsilon_1(\psi_0,\psi_1,\psi_2,n,R,\tau,c_0,A,B)>0$ such that $\epsilon_1\leqslant E_1^{-1}\left(\max\{R/c_0,2 \}\right)^{\frac{n-3}{2}}$.
\newline Considering $\epsilon_0\in(0,\epsilon_1]$ and  $t> (E_1\epsilon_0)^{\frac{2}{n-3}}$, it results
\begin{align*}
	t\geqslant \max\left\{R/c_0,2 \right\}\ \ \mbox{and}\  \ \ln\left(E_1\epsilon_0^2\,t^{3-n}\right)>0.
\end{align*}
Hence, for any $\epsilon_0\in(0,\epsilon_1]$ and any $t> (E_1\epsilon_0)^{\frac{2}{n-3}}$ letting $j\to\infty$ in \eqref{P10}, we find that the lower bound for $\ml{D}_0(t)$ is not finite. Therefore, we conclude the blow-up in finite time when $n=1,2$ and, as byproduct of the previous argument, the upper bound estimate $T(\epsilon_0)\lesssim \epsilon_0^{\frac{2}{n-3}}$ for the lifespan of the local (in time) energy solution $\psi$.

\subsection{Iteration argument when $n=3$}
Let us complete the proof in the case $n=3$. The iteration frame is now given by 
\begin{align}\label{P12} 
	\ml{D}_0(t)\geqslant C_3\int_1^t(R+c_0\eta)^{-1}(\ml{D}_0(\eta))^2\mathrm{d}\eta ,
\end{align} while a first lower bound estimate is provided by \eqref{P11bis}.
 Plugging \eqref{P11bis} into \eqref{P12}, we see that the first iteration produces a logarithmic factor, namely,
\begin{align*}
\ml{D}_0(t)\geqslant C_3C_2^2\epsilon_0^2\int_1^t(R+c_0\eta)^{-1}\mathrm{d}\eta=\frac{C_3C_2^2\epsilon_0^2}{c_0}\ln\left(\frac{R+c_0t}{R+c_0}\right)
\end{align*}
for any $t\geqslant1$. 

As suggested by the first iteration, our blow-up result for $n=3$ will be established through the following sequence of logarithmic-type lower bounds:
\begin{align}\label{P14} 
\ml{D}_0(t)\geqslant\widetilde{Q}_j\epsilon_0^{2^j}\left[\ln\left(\frac{R+c_0 t}{R+c_0}\right)\right]^{\ell_j}
\end{align}
for any $t\geqslant 1$ and any $j\in\mathbb{N}_0$, where $\{\widetilde{Q}_j\}_{j\in\mb{N}_0}$ and $\{\ell_j\}_{j\in\mb{N}_0}$ are sequences of non-negative real numbers that will be determined iteratively. Clearly, thanks to \eqref{P11bis}, we see that \eqref{P14} holds when $j=0$ with $\widetilde{Q}_0=C_2$ and $\ell_0=0$. Using \eqref{P14} in \eqref{P12}, by an elementary integration, we have
\begin{align*}
\ml{D}_0(t)&\geqslant C_3\widetilde{Q}_j^2\epsilon_0^{2^{j+1}}\int_1^t(R+c_0\eta)^{-1}\left[\ln\left(\frac{R+c_0\eta}{R+c_0}\right)\right]^{2\ell_j}\mathrm{d}\eta\\
&\geqslant\frac{C_3\widetilde{Q}_j^2}{c_0(2\ell_j+1)}\epsilon_0^{2^{j+1}}\left[\ln\left(\frac{R+c_0t}{R+c_0}\right)\right]^{2\ell_j+1}
\end{align*}
for any $t\geqslant 1$, which is exactly \eqref{P14} for $j+1$ provided that
\begin{align*}
	\widetilde{Q}_{j+1}:=\frac{C_3}{c_0(2\ell_j+1)}\widetilde{Q}_j^2,\ \ \ell_{j+1}:=2\ell_j+1.
\end{align*}
Using recursively the previous relation for $\ell_j$, we obtain
\begin{align*}
\ell_j=2^j\ell_0+2^j-1=2^j-1,
\end{align*}
which implies  $\widetilde{Q}_j\geqslant C_3c_0^{-1}2^{-j}\widetilde{Q}_{j-1}^2$. Proceeding as in the previous subsection, we find
\begin{align*}
\ln\widetilde{Q}_j\geqslant 2^j\left(\ln\frac{C_3}{c_0}-2\ln 2+\ln\widetilde{Q}_0\right)+j\ln 2+2\ln 2-\ln\frac{C_3}{c_0}.
\end{align*}
Taking $j_1=j_1(n)\in\mb{N}$ to be the smallest non-negative integer such that $j_1\geqslant\ln (C_3/c_0) /\ln 2-2$, for any $j\geqslant j_1$, we obtain
\begin{align*}
	\ln\widetilde{Q}_j\geqslant 2^j\ln\left(\frac{C_2 C_3}{4c_0}\right)=:2^j\ln E_2
\end{align*}
for any $t\geqslant 1$. Combining the last inequality with \eqref{P14} and $\ell_j=2^j-1$, for $j\geqslant j_1$ we arrive at
\begin{align*}
	\ml{D}_0(t)&\geqslant \exp\left(2^j\ln E_2\right)\epsilon_0^{2^j}\left[\ln\left(\frac{R+c_0 t}{R+c_0}\right)\right]^{2^j-1}\notag\\
	&\geqslant\exp\left[2^j\ln\left(E_2\epsilon_0\ln\left(\frac{R+c_0 t}{R+c_0}\right)\right)\right]\left[\ln\left(\frac{R+c_0 t}{R+c_0}\right)\right]^{-1}.\notag
\end{align*} Hence, taking $t\geqslant \max\left\{1,\frac{R^2}{c_0^2}\right\}$ so that $\ln\left(\frac{R+c_0 t}{R+c_0}\right)\geqslant \frac{1}{2}\ln t$, for $j\geqslant j_1$ it results
\begin{align}
	\ml{D}_0(t)
	&\geqslant\exp\left[2^j\ln\left(\frac{E_2}{2}\epsilon_0\ln t\right)\right]\left[\ln\left(\frac{R+c_0 t}{R+c_0}\right)\right]^{-1}.\label{P17}
\end{align}
For $t>\exp(2/ (E_2\epsilon_0))$, we notice that the lower bound in \eqref{P17} is divergent as $j\to\infty$.  We fix $\epsilon_2=\epsilon_2(\psi_0,\psi_1,\psi_2,n,R,\tau,c_0,A,B)>0$ such that $\exp(2/ (E_2\epsilon_2))\geqslant\max\left\{1,\frac{R^2}{c_0^2}\right\}$.
Then, for any  $\epsilon_0\in(0,\epsilon_2]$ and any $t>\exp(2/( E_2\epsilon_0))$, it results 
\begin{align*}
	t\geqslant \max\left\{1,\frac{R^2}{c_0^2}\right\}\ \ \mbox{and}\  \ \ln\left(\frac{E_2}{2}\epsilon_0\ln t\right)>0
\end{align*}
and consequently,  letting $j\to\infty$, we find that  $\ml{D}_0(t)$ cannot be finite. Thus, we completed the proof of Theorem \ref{Thm-Blow-up} when $n=3$ and we obtained  the upper bound estimate $\ln T(\epsilon_0)\lesssim \epsilon_0^{-1}$ for lifespan of  $\psi$.

\section{Final remarks}\label{Sec-Final}
Let us summarize the results that we obtained in this paper in  the next table, which outlines the influence of the viscous dissipation in the nonlinear Cattaneo-type model \eqref{Eq-Cattaneo-Fractional}.

\begin{table}[http]
	\centering	
	\begin{tabular}{cccc}
		\toprule
		Nonlinear Cattaneo-type   & \multirow{2}{*}{Global solvability} & \multirow{2}{*}{Lifespan estimates} & \multirow{2}{*}{\shortstack{Large-time\\ profiles}}\\
		model \eqref{Eq-Cattaneo-Fractional} & & & \\
		\midrule
		\multirow{4}{*}{\shortstack{Viscous case\\ ($0< b\nu\ll 1$)}}  & \multirow{4}{*}{\shortstack{Global existence\\ of small data \\Sobolev solutions}}& \multirow{4}{*}{$T_{\epsilon_0}=\infty$} &\multirow{2}{*}{\shortstack{ Anomalous diffusion\\ if $\alpha\in[0,1/2)$}} \\
		&  & &  \\
				\cmidrule{4-4}
		  & & &\multirow{2}{*}{\shortstack{Diffusion waves\\ if $\alpha\in[1/2,1]$}}  \\
		  & & &\\
		\midrule
		\multirow{4}{*}{\shortstack{Inviscid case\\ ($b\nu=0$)}} & \multirow{4}{*}{\shortstack{Blow-up of the\\ energy solutions \\ in finite time}} & \multirow{4}{*}{$T_{\epsilon_0}\leqslant\begin{cases}
			C\epsilon_0^{-\frac{2}{3-n}}&\mbox{if}\ n=1,2,\\
			\exp(C\epsilon_0^{-1})&\mbox{if}\ n=3.
		\end{cases}$} &\multirow{4}{*}{\diagbox{}{} } \\
		 &  & &\\
		& & &\\
			& & &\\
		\bottomrule
	\end{tabular}
	\caption{The influence of the viscous dissipations}
	\label{Table_3}
\end{table}

\begin{remark}\label{Rem-Regularity}
In the diffusion wave case when $\alpha\in(1/2,1]$, comparing with the regularities of the solution \eqref{D2} for the linearized problem \eqref{Eq-Linear-Cattaneo-Fractional}, and those of the solution \eqref{D1} for the nonlinear problem \eqref{Eq-Cattaneo-Fractional}, we observe a loss of spatial regularity for $\psi_t$ and $\psi_{tt}$. This phenomenon is induced by some problem with the regularity of the nonlinear terms $\nabla\psi\cdot\nabla\psi_t$
in the Cattaneo-type model \eqref{Eq-Cattaneo-Fractional}. As we shown in Section \ref{Sub-GESES-DIFFUSION-WAVE}, by using the representations \eqref{Rep-non} and \eqref{Cn}, we need $\partial_t^{j+1}$-derivatives of the kernel to control $\partial_t^j$-derivatives of the nonlinear part. Thus, it seems interesting to study whether the regularity of the global (in time) solutions to the nonlinear viscous problem \eqref{Eq-Cattaneo-Fractional} can be improved to the one of \eqref{D2} when $\alpha\in(1/2,1]$.
\end{remark}

\begin{remark}\label{Rem_General}
	Thanks to the small viscosity $0<b\nu<2c_0$, the regularity assumptions on the initial data for the anomalous diffusion case $\alpha\in[0,1/2)$ and the diffusion wave case $\alpha=1/2$ are the same. Even when we do not require a small viscosity, Theorems \ref{Thm-01} and \ref{Thm-02}  still hold. However, the asymptotic profiles in the threshold case $\alpha=1/2$ change into some anomalous diffusion functions with singularity $\chi_{\intt}(\xi)|\xi|^{-1}$ as $b\nu>2c_0$; and the function $\ml{F}^{-1}_{\xi\to x}(\mathrm{e}^{-c|\xi|t})t\mb{B}_0$ as $b\nu=2c_0$ 
\end{remark}
\begin{remark}
By following the philosophy in the recent work \cite{Chen-Takeda=2023} and applying several tools from the Fourier analysis, we conjecture that the first\textcolor{green}{-}order profiles $\ml{G}_{k,j}\mb{B}_0$ are the optimal leading terms of the global (in time) solutions $\partial_t^j\psi$ to the nonlinear viscous Cattaneo-type model \eqref{Eq-Cattaneo-Fractional} with weighted $L^1$ data, where $k=1,2,3$ and $j=0,1,2$.
\end{remark}

\begin{remark}
For the inviscid Cattaneo-type model \eqref{Eq-Cattaneo-inviscid}, the question concerning whether or not  global (in time) solutions exist for higher dimensions (assuming the smallness of the Cauchy data), is still open when $n\geqslant 4$.
\end{remark}

\appendix
\section{Tools from the harmonic analysis}\label{Appendix}


In this appendix we collect the results from the harmonic analysis that we used throughout Section \ref{Section_GESDS} to estimate the nonlinear terms in homogeneous Sobolev spaces.

\begin{lemma}\label{fractionalgagliardonirenbergineq} (Fractional Gagliardo-Nirenberg inequality, \cite{Hajaiej-Molinet-Ozawa-Wang-2011})
	Let $p,p_0,p_1\in(1,\infty)$ and $\kappa\in[0,s)$ with $s\in(0,\infty)$. Then, for all $f\in L^{p_0}\cap \dot{H}^{s}_{p_1}$, the following inequality holds:
	\begin{align*}
	\|f\|_{\dot{H}^{\kappa}_{p}}\lesssim\|f\|_{L^{p_0}}^{1-\beta}\|f\|^{\beta}_{\dot{H}^{s}_{p_1}},
	\end{align*}
	where  $\beta=\left(\frac{1}{p_0}-\frac{1}{p}+\frac{\kappa}{n}\right)\big/\left(\frac{1}{p_0}-\frac{1}{p_1}+\frac{s}{n}\right)$ and $ \beta\in\left[\frac{\kappa}{s},1\right]$.
\end{lemma}

\begin{lemma}\label{fractionleibnizrule} (Fractional Leibniz rule, \cite{Grafakos-Oh-2014})
	Let $s\in(0,\infty)$, $r\in[1,\infty]$, $p_1,p_2,q_1,q_2\in(1,\infty]$ satisfy $\frac{1}{r}=\frac{1}{p_1}+\frac{1}{p_2}=\frac{1}{q_1}+\frac{1}{q_2}$. Then, for all $f\in\dot{H}^{s}_{p_1}\cap L^{q_1}$ and $g\in\dot{H}^{s}_{q_2}\cap L^{q_2}$,
	the following inequality holds:
	\begin{align*}
	\|fg\|_{\dot{H}^{s}_{r}}\lesssim \|f\|_{\dot{H}^{s}_{p_1}}\|g\|_{L^{p_2}}+\|f\|_{L^{q_1}}\|g\|_{\dot{H}^{s}_{q_2}}.
	\end{align*}
\end{lemma}

\begin{lemma}\label{fractionembedd} (Fractional Sobolev embedding, \cite{Dabbicco-Ebert-Lucente-2017}) Let $0<2s^*<n<2s$. Then, for all $f\in\dot{H}^{s^*}\cap\dot{H}^s$, the following inequality holds:
	\begin{equation*}
	\|f\|_{L^{\infty}}\lesssim\|f\|_{\dot{H}^{s^*}}+\|f\|_{\dot{H}^s}.
	\end{equation*}
\end{lemma}

\begin{lemma} [Lemma 3.6 in \cite{Palmieri-Reissig=2018}] 
 \label{Sickel lemma}  
Let $s\geqslant 0$ and $q\in (1,\infty)$.  Then we have the following equivalence of norms:
\begin{align*}
\| |D|^s \nabla u\|_{L^q} \approx \sum_{j=1}^n \| |D|^s \partial_{x_j} u\|_{L^q} \approx \| |D|^{s+1}  u\|_{L^q}
\end{align*}
under the assumption, that one of the previous norms exists.
\end{lemma}


\section*{Acknowledgments}

The work was supported by the National Natural Science Foundation of China (Grant 11971497), the Natural Science Foundation of Guangdong Province (Grants 2019B151502041, 2020B1515310004), and the Natural Science Foundation of the Department of Education of Guangdong Province (Grants 2018KZDXM048, 2019KZDXM036, 2020ZDZX3051, 2020TSZK005).   A. Palmieri is the member of the Gruppo Nazionale per l'Analisi Matematica, la Probabilit\`a e le loro Applicazioni (GNAMPA) of the Instituto Nazionale di Alta Matematica (INdAM) and has been supported by INdAM - GNAMPA Project 2021  \textquotedblleft 
 Equazioni dispersive
e dissipative: stime 
e profili asintotici\textquotedblright.

\end{document}